\documentclass{tran-l}

\newtheorem{theorem}{Theorem}[section]
\newtheorem{lemma}[theorem]{Lemma}
\newtheorem{proposition}[theorem]{Proposition}
\newtheorem{corollary}[theorem]{Corollary}

\theoremstyle{definition}

\newtheorem{example}[theorem]{Example}

\theoremstyle{remark}
\newtheorem{remark}[theorem]{Remark}

\numberwithin{equation}{section}

\newcommand{\QED}{\qed}
\newcommand{\n}{\par\noindent}
\newcommand{\sn}{\par\smallskip\noindent}
\newcommand{\mn}{\par\medskip\noindent}

\newcommand{\pars}{\par\smallskip}
\newcommand{\parm}{\par\medskip}
\newcommand{\parb}{\par\bigskip}
\newcommand{\bfind}[1]{{\bf #1}}
\newcommand{\isom}{\simeq}
\newcommand{\bA}{{\bf A}}
\newcommand{\dec}{^d}
\newcommand{\sep}{^{\rm sep}}
\newcommand{\chara}{\mbox{\rm char}\,}
\newcommand{\trdeg}{\mbox{\rm trdeg}\,}
\newcommand{\Gal}{\mbox{\rm Gal}\,}
\newcommand{\rr}{\mbox{\rm rr}\,}
\newcommand{\ic}{\mbox{\rm IC}\,}
\newcommand{\subsetuneq}{\mathrel{\raisebox{.8ex}{\footnotesize%
$\displaystyle\mathop{\subset}_{\not=}$}}}
\newcommand{\cal}{\mathcal}
\newcommand{\eu}{\mathfrak}
\newcommand{\N}{\mathbb N}
\newcommand{\Q}{\mathbb Q}

\newcommand{\Z}{\mathbb Z}
\newcommand{\C}{\mathbb C}
\newcommand{\F}{\mathbb F}
\newcommand{\Fp}{\F_p}

\begin{document}

\title{Value groups, residue fields and bad places of rational
function fields}

\author{Franz-Viktor Kuhlmann}
\address{Mathematical Sciences Group,
University of Saskatchewan,
106 Wiggins Road,
Saskatoon, Saskatchewan, Canada S7N 5E6}
%\curraddr{}
\email{fvk@math.usask.ca}
%\thanks{}

\subjclass[2000]{Primary 12J10; Secondary 12J15, 16W60}
\date{24. 6. 2003}

%\dedicatory{}

\begin{abstract}
We classify all possible extensions of a valuation from a ground field
$K$ to a rational function field in one or several variables over $K$.
We determine which value groups and residue fields can appear, and we
show how to construct extensions having these value groups and residue
fields. In particular, we give several constructions of extensions
whose corresponding value group and residue field extensions are not
finitely generated. In the case of a rational function field $K(x)$ in
one variable, we consider the relative algebraic closure of $K$ in the
henselization of $K(x)$ with respect to the given extension, and we show
that this can be any countably
generated separable-algebraic extension of $K$. In the ``tame case'',
we show how to determine this relative algebraic closure. Finally, we
apply our methods to power series fields and the $p$-adics.
\end{abstract}

\maketitle
\markboth{FRANZ-VIKTOR KUHLMANN}{VALUE GROUPS, RESIDUE FIELDS AND BAD
PLACES}

\section{Introduction}
In this paper, we denote a valued field by $(K,v)$, its value group by
$vK$, and its residue field by $Kv$. When we write $(L|K,v)$ we mean a
field extension $L|K$ endowed with a valuation $v$ on $L$ and its
restriction on $K$.

In many recent applications of valuation theory, valuations on algebraic
function fields play a main role. To mention only a short and incomplete
list of applications and references: local uniformization and resolution
of singularities ([C], [CP], [S], [KU3], [KKU1,2]), model theory of
valued fields ([KU1,2,4]), study of curves via constant reduction
([GMP1,2], [PL]), classification of all extensions of an ordering from a
base field to a rational function field ([KUKMZ]), Gr\"obner bases
([SW], [MOSW1,2]).

In many cases, a basic tool is the classification of
all extensions of a valuation from a base field to a function field.
As the classification of all extensions of a valuation from a field to
an algebraic extension is taken care of by general ramification theory
(cf.\ [E], [KU2]), a crucial step in the classification is the case of
rational function fields. Among the first papers describing valuations
on rational function fields systematically were [M] and [MS]. Since
then, an impressive number of papers have been written about the
construction of such valuations and about their properties; the
following list is by no means exhaustive: [AP], [APZ1-3], [KH1-10],
[KHG1-6], [KHPR], [MO1,2], [MOSW1], [O1-3], [PP], [V]. From the paper
[APZ3] the reader may get a good idea of how MacLane's original approach
has been developed further. Since then, the notion of ``minimal pairs''
has been adopted and studied by several authors (see, e.g., [KHPR]). In
the present paper, we will develop a new approach to this subject. It
serves to determine in full generality which value groups and which
residue fields can possibly occur. This question has recently played a
role in two other papers:

\sn
{\bf 1)} \ In [KU3], we prove the existence of ``bad places'' on
rational function fields of transcendence degree $\geq 2$. These are
places whose value group is not finitely generated, or whose residue
field is not finitely generated over the base field. The existence of
such places has been shown by MacLane and Schilling ([MS]) and by
Zariski and Samuel ([ZS], ch.~VI, \S15, Examples 3 and 4). However, our
approach in [KU3] using Hensel's Lemma seems to be new, and the present
paper contains a further refinement of it. The following theorem of [MS]
and [ZS] is a special case of a result which we will prove by this
refinement:
\begin{theorem}                             \label{badpl}
Let $K$ be any field. Take $\Gamma$ to be any non-trivial ordered
abelian group of finite rational rank $\rho$, and $k$ to be any
countably generated extension of $K$ of finite transcendence degree
$\tau$. Choose any integer $n>\rho+\tau$. Then the rational function
field in $n$ variables over $K$ admits a valuation whose restriction to
$K$ is trivial, whose value group is $\Gamma$ and whose residue field is
$k$.

In particular, every additive subgroup of $\Q$ and every countably
generated algebraic extension of $K$ can be realized as value group and
residue field of a place of the rational function field $K(x,y)|K$
whose restriction to $K$ is the identity.
\end{theorem}
The rational rank of an abelian group $\Gamma$ is the dimension of the
$\Q$-vector space $\Q\otimes_{\Z}\Gamma$. We denote it by $\rr\Gamma$.
It is equal to the cardinality of any maximal set of rationally
independent elements in $\Gamma$.

\pars
Bad places on function fields are indeed bad: the value group or residue
field not being finitely generated constitutes a major hurdle for the
attempt to prove local uniformization or other results which are related
to resolution of singularities (cf.\ [CP]). Another hurdle is the
phenomenon of \bfind{defect} which can appear when the residue
characteristic of a valued field is positive, even if the field
itself seems to be quite simple. Indeed, we will prove in
Section~\ref{sectne}, and by a different method in
Section~\ref{sectMT}:
\begin{theorem}                             \label{piltant}
Let $K$ be any algebraically closed field of positive characteristic.
Then there exists a valuation $v$ on the rational function field
$K(x,y)|K$ whose restriction to $K$ is trivial, such that $(K(x,y),v)$
admits an infinite chain of immediate Galois extensions of degree $p$
and defect $p$.
\end{theorem}
\n
An extension $(L'|L,v)$ of valued fields is called \bfind{immediate} if
the canonical embeddings of $vL$ in $vL'$ and of $Lv$ in $L'v$ are
surjective (which we will express by writing $vL'=vL$ and $L'v=Lv$).
For a finite immediate extension $(L'|L,v)$, its defect is equal to its
degree if and only if the extension of $v$ from $L$ to $L'$ is unique
(or equivalently, $L'|L$ is linearly disjoint from some (or every)
henselization of $(L,v)$).

One of the examples we shall construct for the proof of the above
theorem is essentially the same as in Section 7 of [CP], but we use a
different and more direct approach (while the construction in [CP] is
more intricate since it serves an additional purpose).

\mn
{\bf 2)} \ In [KUKMZ], the classification of all extensions of an
ordering to a rational function field is considered in the context of
power series fields, and the above question is partially answered in
this setting. In the present paper, we will consider the question
without referring to power series fields (see Theorem~\ref{extord}
below).

\pars
During the preparation of [KUKMZ], we found that the construction of an
extension of the valuation $v$ from $K$ to the rational function
field $K(x)$ with prescribed value group $vK(x)$ and residue field
$K(x)v$ is tightly connected with the determination of the relative
algebraic closure of $K$ in a henselization $K(x)^h$ of $K(x)$ with
respect to $v$. In earlier papers, we have introduced the name
``henselian function field'' for the henselizations of valued function
fields (although these are not function fields, unless the valuation is
trivial). In the same vein, one can view the relative algebraic closure
as being the (exact) constant field of the henselian function field
$(K(x)^h|K,v)$. We will call it the \bfind{implicit constant field of
$(K(x)|K,v)$} and denote it by $\ic(K(x)|K,v)$. Clearly, the
henselization $K(x)^h$ depends on the valuation which has been fixed on
the algebraic closure $\widetilde{K(x)}$. So whenever we will talk about
the implicit constant field, we will do it in a setting where the
valuation on $\widetilde{K(x)}$ has been fixed. However, since the
henselization $L^h$ of any valued field $(L,v)$ is unique up to
valuation preserving isomorphism over $L$, the implicit constant field
is unique up to valuation preserving isomorphism over $K$. If $L_0$ is a
subfield of $L$, then $L^h$ contains a (unique) henselization of
$L_0\,$. Hence, $\ic(K(x)|K,v)$ contains a henselization of $K$ and is
itself henselian. Further, $L^h|L$ is a separable-algebraic extension;
thus, $K(x)^h|K$ is separable. Therefore, $\ic(K(x)|K,v)$ is a
separable-algebraic extension of $K$.

In the present paper, we answer the above question on value groups and
residue fields by determining which prescribed separable-algebraic
extensions of $K$ can be realized as implicit constant fields. The
following result shows in particular that every countably generated
separable-algebraic extension of a henselian base field can be
realized:
\begin{theorem}                             \label{MT}
Let $(K_1|K,v)$ be a countable separable-algebraic extension of
non-trivially valued fields. Then there is an extension of $v$ from
$K_1$ to the algebraic closure $\widetilde{K_1(x)}=\widetilde{K(x)}$ of
the rational function field $K(x)$ such that, upon taking henselizations
in $(\widetilde{K(x)},v)$,
\begin{equation}
K_1^h\;=\; \ic(K(x)|K,v)\;.
\end{equation}
\end{theorem}

In Section~\ref{sectbc} we will introduce a basic classification
(``value-transcendental'' -- ``residue-transcendental'' --
``valuation-algebraic'') of all possible extensions of $v$ from $K$ to
$K(x)$. In Section~\ref{sectpure} we introduce a class of extensions
$(K(x)|K,v)$ for which $\ic(K(x)|K,v)=K^h$ holds. Building on this, we
prove Theorem~\ref{MT} in Section~\ref{sectMT}. In fact, we prove a
more detailed version: we show under which additional conditions
the extension can be chosen in a prescribed class of the basic
classification. This yields the following

\begin{theorem}                             \label{MT1var}
Take any valued field $(K,v)$, an ordered abelian group extension
$\Gamma_0$ of $vK$ such that $\Gamma_0/vK$ is a torsion group, and an
algebraic extension $k_0$ of $Kv$. Further, take $\Gamma$ to be the
abelian group $\Gamma_0 \oplus\Z$ endowed with any extension of the
ordering of $\Gamma_0\,$.

Assume first that $\Gamma_0/vK$ and $k_0|Kv$ are finite. If $v$ is
trivial on $K$, then assume in addition that $k_0|Kv$ is simple. Then
there is an extension of $v$ from $K$ to the rational function field
$K(x)$ which has value group $\Gamma$ and residue field $k_0$. If $v$ is
non-trivial on $K$, then there is also an extension which has value
group $\Gamma_0$ and as residue field a rational function field in one
variable over $k_0\,$.

Now assume that $v$ is non-trivial on $K$ and that $\Gamma_0/vK$ and
$k_0|Kv$ are countably generated. Suppose that at least one of them is
infinite or that $(K,v)$ admits an immediate transcendental extension.
Then there is an extension of $v$ from $K$ to $K(x)$ which has value
group $\Gamma_0$ and residue field $k_0$.
\end{theorem}

Here is the converse:
\begin{theorem}                             \label{conv1}
Let $(K(x)|K,v)$ be a valued rational function field. Then one and
only one of the following three cases holds:
\sn
1) \ $vK(x)\simeq \Gamma_0\oplus \Z$, where $\Gamma_0|vK$ is a finite
extension of ordered abelian groups, and $K(x)v|Kv$ is finite;\n
2) \ $vK(x)/vK$ is finite, and $K(x)v$ is a rational function field in
one variable over a finite extension of $Kv$;\n
3) \ $vK(x)/vK$ is a torsion group and $K(x)v|Kv$ is algebraic.
\sn
In all cases, $vK(x)/vK$ is countable and $K(x)v|Kv$ is countably
generated.
\end{theorem}

In 2), we use a fact which was proved by J.~Ohm [O2] and is known as the
``Ruled Residue Theorem'': {\it If $K(x)v|Kv$ is transcendental, then
$K(x)v$ is a rational function field in one variable over a finite
extension of $Kv$.} For the countability assertion, see
Theorem~\ref{count}.

In Section~\ref{sectae} we give an explicit description of all possible
extensions of $v$ from $K$ to $K(x)$ (Theorem~\ref{allext}).

\pars
Theorem~\ref{MT1var} is used in the proof of our next theorem:
\begin{theorem}                             \label{MTsevvar}
Let $(K,v)$ be any valued field, $n,\rho,\tau$ non-negative integers,
$n\geq 1$, $\Gamma\ne\{0\}$ an ordered abelian group extension of $vK$
such that $\Gamma/vK$ is of rational rank $\rho$, and $k|Kv$ a field
extension of transcendence degree $\tau$.
\sn
{\bf Part A.} \ Suppose that $n>\rho+\tau$ and that
\sn
A1) \ $\Gamma/vK$ and $k|Kv$ are countably generated,
\n
A2) \ $\Gamma/vK$ or $k|Kv$ is infinite.
\sn
Then there is an extension of $v$ to the rational
function field $K(x_1,\ldots,x_n)$ in $n$ variables such that
\begin{equation}                            \label{rfwvgarf}
vK(x_1,\ldots,x_n)\>=\>\Gamma\;\;\mbox{\ \ and\ \ }\;\;
K(x_1,\ldots,x_n)v\>=\>k\;.
\end{equation}
\mn
{\bf Part B.} \ Suppose that $n\geq\rho+\tau$ and that
\sn
B1) \ $\Gamma/vK$ and $k|Kv$ are finitely generated,
\n
B2) \ if $v$ is trivial on $K$, $n=\rho+\tau$ and $\rho=1$, then $k$
is a simple algebraic extension of a rational function field in $\tau$
variables over $Kv$ (or of $Kv$ itself if $\tau=0$), or a rational
function field in one variable over a finitely generated field extension
of $Kv$ of transcendence degree $\tau-1$,
\n
B3) \ if $n=\tau$, then $k$ is a rational function field
in one variable over a finitely generated field extension of $Kv$ of
transcendence degree $\tau-1$,
\n
B4) \ if $\rho=0=\tau$, then there is an immediate extension of $(K,v)$
which is either infinite separable-algebraic or of transcendence degree
at least $n$.
\sn
Then again there is an extension of $v$ to
$K(x_1,\ldots,x_n)$ such that (\ref{rfwvgarf}) holds.
\end{theorem}

Theorem~\ref{badpl} is the special case of Part A for $v$ trivial on
$K$. The following converse holds:
\begin{theorem}                             \label{conv2}
Let $n\geq 1$ and $v$ be a valuation on the rational function field
$F=K(x_1,\ldots,x_n)$. Set $\rho=\rr vF/vK$ and $\tau=\trdeg
Fv|Kv$. Then $n\geq\rho+\tau$, $vF/vK$ is countable, and $Fv|Kv$ is
countably generated.

If $n=\rho+\tau$, then $vF/vK$ is finitely generated and $Fv|Kv$ is a
finitely generated field extension. Assertions B2) and B3) of
Theorem~\ref{MTsevvar} hold for $k=Fv$, and if $\rho=0= \tau$, then
there is an immediate extension of $(\tilde{K},v)$ of transcendence
degree~$n$ (for any extension of $v$ from $K$ to $\tilde{K}$).
\end{theorem}

\n
There is a gap between Theorem~\ref{MTsevvar} and this converse for the
case of $\rho=0=\tau$, as the former talks about $K$ and the latter
talks about the algebraic closure $\tilde{K}$ of $K$. This gap can be
closed if $(K,v)$ has residue characteristic 0 or is a Kaplansky field;
because the maximal immediate extension of such fields is unique up to
isomorphism, one can show that $\tilde{K}$ can be replaced by $K$. But
in the case where $(K,v)$ is not such a field, we do not know enough
about the behaviour of maximal immediate extensions under algebraic
field extensions. This question should be considered in future research.

\pars
A valuation on an ordered field is called \bfind{convex} if the
associated valuation ring is convex.
%(this implies certain properties of compatibility with the ordering).
For the case of ordered fields with convex valuations, we can derive
from Theorem~\ref{MTsevvar} the existence of convex extensions of the
valuation with prescribed value groups and residue fields in the frame
given by Theorem~\ref{conv2}, provided that a natural additional
condition for the residue fields is satisfied:
\begin{theorem}                             \label{extord}
In the setting of Theorem~\ref{MTsevvar}, assume in addition that $K$ is
ordered and that $v$ is convex w.r.t.\ the ordering. Assume further that
$k$ is equipped with an extension of the ordering induced by $<$ on
$Kv$. Then this extension can be lifted through $v$ to $K(x_1,\ldots,
x_n)$ in such a way that the lifted ordering extends $<$. It follows
that $v$ is convex w.r.t.\ this lifted ordering on $K(x_1,\ldots,x_n)$.
\end{theorem}

\parm
In Section~\ref{sectkrseq} we shall introduce ``homogeneous
sequences''. In the ``tame case'', they can be used to determine the
implicit constant field of a valued rational function field, and also to
characterize this ``tame case''. In Section~\ref{sectappl} we shall show
how to apply our results to power series, in the spirit of [MS] and
[ZS] (Theorem~\ref{powser}). We will also use our approach to give
proofs of two well known facts in p-adic algebra: that the algebraic
closure of $\Q_p$ is not complete and that its completion is not
maximal.

Finally, let us mention that we use our criteria for $\ic(K(x)|K,v)=K^h$
in Section~\ref{sectpure} to give an example for the following fact:
{\it Suppose that $K$ is relatively algebraically closed in a henselian
valued field $(L,v)$ such that $vL/vK$ is a torsion group. Then it is
not necessarily true that $vL=vK$, even if $v$ has residue
characteristic $0$.}

\parb
I would like to thank Murray Marshall and Salma Kuhlmann for the very
inspiring joint seminar; without this seminar, this paper would not have
been written. I also feel very much endebted to Sudesh Kaur Khanduja for
many ideas she has shared with me. Finally, I would like to thank Roland
Auer for finding some mistakes and asking critical questions.

%
%Ä Ä Ä Ä Ä Ä Ä Ä Ä Ä Ä Ä Ä Ä Ä Ä Ä Ä Ä Ä Ä Ä Ä Ä Ä Ä Ä Ä Ä Ä Ä Ä Ä Ä -
%
\section{Notation and valuation theoretical preliminaries}\label{sectprel}
For an arbitrary field $K$, will will denote by $K\sep$ the
separable-algebraic closure of $K$, and by $\tilde{K}$ the algebraic
closure of $K$. By $\Gal K$ we mean the absolute Galois group
$\Gal(\tilde{K}|K)= \Gal (K\sep|K)$. For a valuation $v$ on $K$,
we let ${\cal O}_K$ denote the valuation ring of $v$ on $K$.

Every finite extension $(L|K,v)$ of valued fields
satisfies the {\bf fundamental inequality} (cf.\ [E]):
\begin{equation}                             \label{fiq}
n\>\geq\>\sum_{i=1}^{\rm g} {\rm e}_i {\rm f}_i
\end{equation}
where $n=[L:K]$ is the degree of the extension, $v_1,\ldots,v_{\rm g}$
are the distinct extensions of $v$ from $K$ to $L$, ${\rm e}_i=(v_i
L:vK)$ are the respective ramification indices and ${\rm f}_i=
[Lv_i:Kv]$ are the respective inertia degrees. Note that ${\rm g}=1$
if $(K,v)$ is henselian.

\pars
In analogy to field theory, an extension $\Gamma\subset \Delta$ of
abelian groups will also be written as $\Delta|\Gamma$, and it will be
called \bfind{algebraic} if $\Delta/\Gamma$ is a torsion group. The
fundamental inequality implies the following well known
fact:
\begin{lemma}                               \label{fin}
If $(L|K,v)$ is finite, then so are $vL/vK$ and $Lv|Kv$.
If $(L|K,v)$ is algebraic, then so are $vL/vK$ and $Lv|Kv$.
\end{lemma}

\parm
Given two subextensions $M|K$ and $L|K$ within a fixed extension $N|K$,
the {\bf field compositum} $M.L$ is defined to be the smallest subfield
of $N$ which contains both $M$ and $L$. If $L|K$ is algebraic, the
compositum is uniquely determined by taking $N=\tilde{M}$ and specifying
a $K$-embedding of $L$ in $\tilde{M}$.

\begin{lemma}                               \label{persimm}
Let $(M|K,v)$ be an immediate extension of valued fields, and $(L|K,v)$
a finite extension such that $[L:K]=(vL:vK)[Lv:Kv]$. Then for every
$K$-embedding of $L$ in $\tilde{M}$ and every extension of $v$ from $M$
to $\tilde{M}$, the extension $(M.L|L,v)$ is immediate.
\end{lemma}
\begin{proof}
Via the embedding, we identify $L$ with a subfield of $\tilde{M}$.
Pick any extension of $v$ from $M$ to $\tilde{M}$. This will also be
an extension of $v$ from $L$ to $\tilde{M}$ because by the fundamental
inequality, the extension of $v$ from $K$ to $L$ is unique. We consider
the extension $(M.L|M,v)$. It is clear that $vL\subseteq v(M.L)$ and
$Lv\subseteq (M.L)v$; therefore, $(v(M.L):vK)\geq
(vL:vK)$ and $[(M.L)v:Kv]\geq [Lv:Kv]$. Since $(M|K,v)$ is immediate,
we have
\begin{eqnarray*}
[M.L:M] & \geq & (v(M.L):vM)[(M.L)v:Mv]\>=\>(v(M.L):vK)[(M.L)v:Kv]\\
 & \geq & (vL:vK)[Lv:Kv]\>=\>[L:K]\>\geq\>[M.L:M]\;.
\end{eqnarray*}
This shows that $(v(M.L):vK)=(vL:vK)$ and $[(M.L)v:Kv]=[Lv:Kv]$, that
is, $v(M.L)=vL$ and $(M.L)v=Lv$.
\end{proof}

%
%ÄÄÄÄÄÄÄÄÄÄÄÄÄÄÄÄÄÄÄÄÄÄÄÄÄÄÄÄÄÄÄÄÄÄÄÄÄÄÄÄÄÄÄÄÄÄÄÄÄÄÄÄÄÄÄÄÄÄÄÄÄÄÄÄÄÄÄÄÄÄÄ
%
\subsection{Pseudo Cauchy sequences}
We assume the reader to be familiar with the theory of pseudo Cauchy
sequences as presented in [KA]. Recall that a pseudo Cauchy sequence
$\bA=(a_{\nu})_{\nu<\lambda}$ in $(K,v)$ (where $\lambda$ is some limit
ordinal) is \bfind{of transcendental type} if for every $g(x)\in K(x)$,
the value $vg(a_{\nu})$ is eventually constant, that is, there is some
$\nu_0<\lambda$ such that
\begin{equation}                            \label{fixvalg}
vg(a_{\nu})\>=\>vg(a_{\nu_0})\;\;\;\mbox{ for all \ }
\nu\geq\nu_0 \>,\> \nu<\lambda\;.
\end{equation}
Otherwise, \bA\ if \bfind{of algebraic type}.

Take a pseudo Cauchy sequence \bA\ in $(K,v)$ of transcendental type.
We define an extension $v_{\bA}$ of $v$ from $K$ to the rational
function field $K(x)$ as follows. For each $g(x)\in K[x]$, we choose
$\nu_0<\lambda$ such that (\ref{fixvalg}) holds. Then we set
\[v_{\bA} \,g(x)\>:=\>vg(a_{\nu_0})\;.\]
We extend $v_{\bA}$ to $K(x)$ by setting $v_{\bA}(g/h):=
v_{\bA}g-v_{\bA} h$. The following is Theorem~2 of [KA]:
\begin{theorem}                             \label{Ka2}
Let \bA\ be a pseudo Cauchy sequence in $(K,v)$ of transcendental type.
Then $v_{\bA}$ is a valuation on the rational function field $K(x)$. The
extension $(K(x)|K,v_{\bA})$ is immediate, and $x$ is a pseudo limit of
\bA\ in $(K(x),v_{\bA})$. If $(K(y),w)$ is any other valued extension of
$(K,v)$ such that $y$ is a pseudo limit of \bA\ in $(K(y),w)$, then
$x\mapsto y$ induces a valuation preserving $K$-isomorphism from
$(K(x),v_{\bA})$ onto $(K(y),w)$.
\end{theorem}

From this theorem we deduce:
\begin{lemma}                               \label{limtpcs}
Suppose that in some valued field extension of $(K,v)$,
$x$ is the pseudo limit of a pseudo Cauchy sequence in $(K,v)$ of
transcendental type. Then $(K(x)|K,v)$ is immediate and $x$ is
transcendental over~$K$.
\end{lemma}
\begin{proof}
Assume that $(a_\nu)_{\nu<\lambda}$ is a pseudo Cauchy sequence in
$(K,v)$ of transcendental type. Then by Theorem~\ref{Ka2} there is an
immediate extension $w$ of $v$ to the rational function field $K(y)$
such that $y$ becomes a pseudo limit of $(a_\nu)_{\nu<\lambda}\,$;
moreover, if also $x$ is a pseudo limit of $(a_\nu)_{\nu<\lambda}$ in
$(K(x),v)$, then $x\mapsto y$ induces a valuation preserving isomorphism
from $K(x)$ onto $K(y)$ over $K$. Hence, $(K(x)|K,v)$ is immediate and
$x$ is transcendental over~$K$.
\end{proof}

\begin{lemma}                               \label{persist}
A pseudo Cauchy sequence of transcendental type in a valued field
remains a pseudo Cauchy sequence of transcendental type in every
algebraic valued field extension of that field.
\end{lemma}
\begin{proof}
Assume that $(a_\nu)_{\nu<\lambda}$ is a pseudo Cauchy sequence in
$(K,v)$ of transcendental type and that $(L|K,v)$ is an algebraic
extension. If $(a_\nu)_{\nu<\lambda}$ were of algebraic type over
$(L,v)$, then by Theorem~3 of [KA] there would be an algebraic
extension $L(y)|L$ and an immediate extension of $v$ to $L(y)$
such that $y$ is a pseudo limit of $(a_\nu)_{\nu<\lambda}$ in
$(L(y),v)$. But then, $y$ is also a pseudo limit of
$(a_\nu)_{\nu<\lambda}$ in $(K(y),v)$. Hence by the foregoing lemma,
$y$ must be transcendental over $K$. This is a contradiction to the
fact that $L(y)|L$ and $L|K$ are algebraic.
\end{proof}

%
%Ä - Ä Ä Ä Ä Ä Ä Ä Ä Ä Ä Ä Ä Ä Ä Ä Ä Ä Ä Ä Ä Ä Ä Ä Ä Ä Ä Ä Ä Ä Ä Ä Ä Ä
%
\subsection{Valuation independence}
For the easy proof of the following lemma, see [B], chapter VI,
\S10.3, Theorem~1.
\begin{lemma}                                      \label{prelBour}
Let $(L|K,v)$ be an extension of valued fields. Take elements $x_i,y_j
\in L$, $i\in I$, $j\in J$, such that the values $vx_i\,$, $i\in I$,
are rationally independent over $vK$, and the residues $y_jv$, $j\in
J$, are algebraically independent over $Kv$. Then the elements
$x_i,y_j$, $i\in I$, $j\in J$, are algebraically independent over $K$.

Moreover, if we write
\[f\>=\> \displaystyle\sum_{k}^{} c_{k}\,
\prod_{i\in I}^{} x_i^{\mu_{k,i}} \prod_{j\in J}^{} y_j^{\nu_{k,j}}\in
K[x_i,y_j\mid i\in I,j\in J]\]
in such a way that for every $k\ne\ell$
there is some $i$ s.t.\ $\mu_{k,i}\ne\mu_{\ell,i}$ or some $j$ s.t.\
$\nu_{k,j}\ne\nu_{\ell,j}\,$, then
%Moreover, if $\>f\>=\> \displaystyle\sum_{k}^{} c_{k}\, \prod_{i\in
%I}^{} x_i^{\mu_{k,i}} \prod_{j\in J}^{} y_j^{\nu_{k,j}}\in K[x_i,y_j\mid
%i\in I,j\in J]$, then
\begin{equation}                            \label{value}
vf\>=\>\min_k\, v\,c_k \prod_{i\in I}^{}
x_i^{\mu_{k,i}}\prod_{j\in J}^{} y_j^{\nu_{k,j}}\>=\>
\min_k\, vc_k\,+\,\sum_{i\in I}^{} \mu_{k,i} v x_i\;.
\end{equation}
That is, the value of the polynomial $f$ is equal to the least of the
values of its monomials. In particular, this implies:
\begin{eqnarray*}
vK(x_i,y_j\mid i\in I,j\in J) & = & vK\oplus\bigoplus_{i\in I}
\Z vx_i\\
K(x_i,y_j\mid i\in I,j\in J)v & = & Kv\,(y_jv\mid j\in J)\;.
\end{eqnarray*}
Moreover, the valuation $v$ on $K(x_i,y_j\mid i\in I,j\in J)$ is
uniquely determined by its restriction to $K$, the values $vx_i$ and
the residues $y_jv$.
\parm
Conversely, if $(K,v)$ is any valued field and we assign to the $vx_i$
any values in an ordered group extension of $vK$ which are rationally
independent, then (\ref{value}) defines a valuation on $L$, and the
residues $y_jv$, $j\in J$, are algebraically independent over $Kv$.
\end{lemma}

\begin{corollary}                              \label{fingentb}
Let $(L|K,v)$ be an extension of finite transcendence degree of valued
fields. Then
\begin{equation}                            \label{wtdgeq}
\trdeg L|K \>\geq\> \trdeg Lv|Kv \,+\, \rr (vL/vK)\;.
\end{equation}
If in addition $L|K$ is a function field and if equality holds in
(\ref{wtdgeq}), then the extensions $vL| vK$ and $Lv|Kv$ are finitely
generated.
\end{corollary}
\begin{proof}
Choose elements $x_1,\ldots,x_{\rho},y_1,\ldots,y_{\tau}\in L$ such that
the values $vx_1,\ldots,vx_{\rho}$ are rationally independent over $vK$
and the residues $y_1v,\ldots,y_{\tau} v$ are algebraically independent
over $Kv$. Then by the foregoing lemma, $\rho+\tau\leq\trdeg L|K$. This
proves that $\trdeg Lv|Kv$ and the rational rank of $vL/vK$ are finite.
Therefore, we may choose the elements $x_i,y_j$ such that $\tau=\trdeg
Lv|Kv$ and $\rho=\rr (vL/vK)$ to obtain inequality (\ref{wtdgeq}).

Assume that equality holds in (\ref{wtdgeq}). This means that $L$ is
an algebraic extension of $L_0:=K(x_1,\ldots, x_{\rho},y_1,\ldots,
y_{\tau})$. Since $L|K$ is finitely generated, it follows that $L|L_0$
is finite; hence by Lemma~\ref{fin}, also $vL/vL_0$ and $Lv| L_0v$ are
finite. Since already $v L_0|v K$ and $L_0v|Kv$ are finitely generated
by the foregoing lemma, it follows that also $vL|vK$ and $Lv|Kv$ are
finitely generated.
\end{proof}

\pars
The algebraic analogue of the transcendental case discussed in
Lemma~\ref{prelBour} is the following lemma (see [R] or [E]):
\begin{lemma}                               \label{algvind}
Let $(L|K,v)$ be an extension of valued fields. Take $\eta_i\in L$ such
that $v\eta_i\,$, $i\in I$, belong to distinct cosets modulo $vK$.
Further, take $\vartheta_j \in {\cal O}_L\,$, $j\in J$, such that
$\vartheta_j v$ are $Kv$-linearly independent. Then the elements
$\eta_i\vartheta_j\,$, $i\in I$, $j\in J$, are $K$-linearly independent,
and for every choice of elements $c_{ij} \in K$, only finitely many of
them nonzero, we have that
\[v\sum_{i,j} c_{ij} \eta_i\vartheta_j\>=\> \min_{i,j}\, v c_{ij}
\eta_i\vartheta_j\>=\> \min_{i,j}\, (vc_{ij}\,+\,v\eta_i)\;.\]
If the elements $\eta_i\vartheta_j$ form a $K$-basis of $L$, then
$v\eta_i\,$, $i\in I$, is a system of representatives of the
cosets of $vL$ modulo $vK$, and $\vartheta_j v$, $j\in J$, is a
basis of $Lv|Kv$.
\end{lemma}
%(In this case, the elements $\eta_i\vartheta_j$ are said to be
%\bfind{valuation independent} over $K$.)

The following is an application which is important for our description
of all possible value groups and residue fields of valuations on $K(x)$.
The result has been proved with a different method in [APZ3]
(Corollary~5.2); cf.\ Remark~\ref{limitrem} in Section~\ref{sectbc}.

\begin{theorem}                             \label{count}
Let $K$ be any field and $v$ any valuation of the rational function
field $K(x)$. Then $vK(x)/vK$ is countable, and $K(x)v|Kv$ is
countably generated.
\end{theorem}
\begin{proof}
Since $K(x)$ is the quotient field of $K[x]$, we have that $vK(x)=vK[x]-
vK[x]$. Hence, to show that $vK(x)/vK$ is countable, it suffices to show
that the set $\{\alpha+vK\mid \alpha\in vK[x]\}$ is countable. If this
were not true, then by Lemma~\ref{algvind} (applied with $J=\{1\}$ and
$\vartheta_1=1$), we would have that $K[x]$ contains uncountably many
$K$-linearly independent elements. But this is not true, as $K[x]$
admits the countable $K$-basis $\{x^i\mid i\geq 0\}$.

Now assume that $K(x)v|Kv$ is not countably generated. Then by
Corollary~\ref{fingentb}, $K(x)v|Kv$ must be algebraic. It also follows
that $K(x)v$ has uncountable $Kv$-dimension. Pick an uncountable set
$\kappa$ and elements $f_i(x)/g_i(x)$, $i\in \kappa$, with $f_i(x),
g_i(x)\in K[x]$ and $vf_i(x)=vg_i(x)$ for all $i$, such that their
residues are $Kv$-linearly independent. As $vK(x)/vK$ is countable,
there must be some uncountable subset $\lambda\subset\kappa$ such that
for all $i\in\lambda$, the values $vf_i(x)=vg_i(x)$ lie in the same
coset modulo $vK$, say $vh(x)+vK$ with $h(x)\in K[x]$. The residues
$(f_i(x)/g_i(x))v$, $i\in \lambda$, generate an algebraic extension of
uncountable dimension. Choosing suitable elements $c_i\in K$ such
that
\[vc_if_i(x) \>=\> vh(x) \>=\> vc_ig_i(x)\;,\]
we can write
\[\frac{f_i(x)}{g_i(x)}\>=\>\frac{c_if_i(x)}{h(x)}\cdot
\frac{h(x)}{c_ig_i(x)} \>=\> \frac{c_if_i(x)}{h(x)}\cdot
\left(\frac{c_ig_i(x)}{h(x)}\right)^{-1}\]
for all $i\in \lambda$. Therefore,
\[\frac{f_i(x)}{g_i(x)}\,v \>=\>
\left(\frac{c_if_i(x)}{h(x)}\,v\right)\,
\cdot \left(\frac{c_ig_i(x)}{h(x)}\,v\right)^{-1}\]
for all $i\in \lambda$. In order that these elements generate an
algebraic extension of $Kv$ of uncountable dimension, the same must
already be true for the elements $(c_if_i(x)/h(x))v$, $i\in \lambda$, or
for the elements $(c_ig_i(x)/h(x))v$, $i\in \lambda$. It follows that at
least one of these two sets contains uncountably many $Kv$-linearly
independent elements. But then by Lemma~\ref{algvind} (applied with
$I=\{1\}$ and $\eta_1=1$), there are uncountably many $K$-linearly
independent elements in the set
\[\frac{1}{h(x)}\,K[x]\]
and hence also in $K[x]$, a contradiction.
\end{proof}

\pars
Finally, let us mention the following lemma which combines the algebraic
and the transcendental case. We leave its easy proof to the reader.
\begin{lemma}                               \label{algtransc}
Let $(L|K,v)$ be an extension of valued fields. Take $x\in L$. Suppose
that for some $e\in\N$ there exists an element $d\in K$ such that
$vdx^e=0$ and $dx^ev$ is transcendental over $Kv$. Let $e$ be minimal
with this property. Then for every $f=c_nx^n+\ldots+c_0\in K[x]$,
\[vf\>=\> \min_{1\leq i\leq n} vc_ix^i\;.\]
Moreover, $K(x)v=Kv(dx^ev)$ is a rational function field over $Kv$, and
we have
\[vK(x)\>=\>vK+\Z vx \mbox{ \ \ with \ \ } (vK(x):vK)\>=\>e\;.\]
\end{lemma}

%
%Ä - Ä Ä Ä Ä Ä Ä Ä Ä Ä Ä Ä Ä Ä Ä Ä Ä Ä Ä Ä Ä Ä Ä Ä Ä Ä Ä Ä Ä Ä Ä Ä Ä Ä
%
\subsection{Construction of valued field extensions with prescribed
value groups and residue fields}             \label{sectconstrvf}
In this section, we will deal with the following problem. Suppose that
$(K,v)$ is a valued field, $\Gamma|vK$ is an extension of ordered
abelian groups and $k|Kv$ is a field extension. Does there exist an
extension $(L|K,v)$ of valued fields such that $vL=\Gamma$ and $Lv=k$?
We include the case of $(K,v)$ being trivially valued; this amounts to
the construction of a valued field with given value group and residue
field. Throughout, we use the well known fact that if $(K,v)$ is any
valued field and $L$ is any extension field of $K$, then there is at
least one extension of $v$ to $L$ (cf.\ [E], [R]).

Let us adjust the following notion to our purposes. Usually, when one
speaks of an \bfind{Artin-Schreier extension} then one means an
extension of a field $K$ generated by a root of an irreducible
polynomial of the form $X^p-X-c$, provided that $p=\chara K$. We will
replace this by the weaker condition ``$p=\chara Kv$''. In fact, such
extensions also play an important role in the mixed characteristic case,
where $\chara K=0$.

Every {\bf Artin-Schreier polynomial} $X^p-X-c$ is separable since its
derivative does not vanish. The following is a simple but very useful
observation:
\begin{lemma}                               \label{vASr}
Let $(K,v)$ be a valued field and $c\in K$ such that $vc<0$. If
$a\in\tilde{K}$ such that $a^p-a=c$, then for every extension of $v$
from $K$ to $K(a)$,
\[0\> >\>v(a^p-c)\>=\>va\> >\>pva\>=\>vc\;.\]
\end{lemma}
\begin{proof}
Take any extension of $v$ from $K$ to $K(a)$. Necessarily, $va<0$ since
otherwise, $\infty=v(a^p-a-c)= \min\{pva,va,vc\}=vc$, a contradiction.
It follows that $va^p=pva<va$ and thus,
\[vc\>=\>\min\{v(a^p-a-c),v(a^p-a)\}\>=\>v(a^p-a)\>=\>\min\{pva,va\}
\>=\> pva\;.\]
\end{proof}

\begin{lemma}                               \label{constrva}
Let $(K,v)$ be a non-trivially valued field, $p$ a prime and $\alpha$ an
element of the divisible hull of $vK$ such that $p\alpha\in vK$,
$\alpha\notin vK$. Choose an element $a\in \tilde{K}$ such that $a^p\in
K$ and $va^p= p\alpha$. Then $v$ extends in a unique way from $K$ to
$K(a)$. It satisfies
\begin{equation}                            \label{dconstrva}
va=\alpha\>,\;\;\; [K(a):K]\>=\>(vK(a):vK)\>,\;\;\;
vK(a)\>=\>vK+\Z\alpha \;\mbox{\ \ and\ \ }\; K(a)v\>=\>Kv\;.
\end{equation}
If $\chara K=\chara Kv=p$, then this extension $K(a)|K$ is purely
inseparable. On the other hand, if $\chara Kv=p$, then there is
always an Artin-Schreier extension $K(a)|K$ with properties
(\ref{dconstrva}); if $\alpha<0$, then $a$ itself can be chosen
to be the root of an Artin-Schreier polynomial over $K$.
\end{lemma}
\begin{proof}
Take $c\in K$ such that $vc=p\alpha$ and $a\in \tilde{K}$ such that
$a^p=c$. Choose any extension of $v$ from $K$ to $K(a)$. Then
$pva=vc=p\alpha$. Consequently, $va= \alpha$ and $(vK(a):vK)\geq
(vK+\Z\alpha:vK)=p$. On the other hand, the fundamental inequality
(\ref{fiq}) shows that
\[p=[K(a):K]\geq (vK(a):vK)\cdot [K(a)v:Kv] \geq (vK(a):vK)
\geq p\;.\]
Hence, equality holds everywhere, and we find that $(vK(a):vK)=p$ and
$[K(a)v:Kv]=1$. That is, $vK(a)=vK+\Z\alpha$ and $K(a)v= Kv$. Further,
the fundamental inequality implies that the extension of $v$ from $K$ to
$K(a)$ is unique.

Now suppose that $\chara Kv=p$. Choose $c\in K$ such that $vc=-|p\alpha|
<0$. By the foregoing lemma, every root $b$ of the Artin-Schreier
polynomial $X^p-X-c$ must satisfy $pvb=vc$. Now we set $a=b$ if
$\alpha<0$, and $a=1/b$ if $\alpha>0$ (but note that then $1/a$ is in
general not the root of an Artin-Schreier polynomial). Then as before
one shows that (\ref{dconstrva}) holds.
\end{proof}

For $f\in {\cal O}_K [X]$, we define the \bfind{reduction} $fv\in Kv[X]$
to be the polynomial obtained from $f$ through replacing every
coefficient by its residue.

\begin{lemma}                               \label{constrra}
Let $(K,v)$ be a valued field and $\zeta$ an element of the algebraic
closure of $Kv$. Choose a monic polynomial $f\in {\cal O}_K [X]$
whose reduction $fv$ is the minimal polynomial of $\zeta$ over
$Kv$. Further, choose a root $b\in \tilde{K}$ of $f$. Then there
is a unique extension of $v$ from $K$ to $K(b)$ and a corresponding
extension of the residue map such that
\begin{equation}                            \label{degvgrf}
bv=\zeta\>,\;\;\; [K(b):K]\>=\> [Kv(\zeta):Kv]\>,\;\;\;
vK(b)\>=\>vK \;\mbox{\ \ and\ \ }\; K(b)v\>=\>Kv(\zeta)\;.
\end{equation}
In all cases, $f$ can be chosen separable, provided that the
valuation $v$ is non-trivial. On the other hand, if $\chara K=
\chara Kv=p>0$ and $\zeta$ is purely inseparable over $Kv$,
then $b$ can be chosen purely inseparable over $K$.

If $v$ is non-trivial, $\chara Kv=p>0$ and $\zeta^p\in Kv$, $\zeta\notin
Kv$, then there is also an Artin-Schreier extension $K(b)|K$ such that
(\ref{degvgrf}) holds and $db$ is the root of an Artin-Schreier
polynomial over $K$, for a suitable $d\in K$.
\end{lemma}
\begin{proof}
We choose an extension of $v$ from $K$ to $K(b)$. Since $f$ is monic
with integral coefficients, $b$ must also be integral for this
extension, and $bv$ must be a root of $fv$. We may compose the residue
map with an isomorphism in $\Gal Kv$ which sends this root to $\zeta$.
Doing so, we obtain a residue map (still associated with $v$) that
satisfies $bv=\zeta$. Now $\zeta\in K(b)v$ and consequently,
$[K(b)v:Kv]\geq [Kv(\zeta):Kv] =\deg fv=\deg f$. On the other hand, the
fundamental inequality shows that
\[\deg f= [K(b):K]\geq (vK(b):vK)\cdot [K(b)v:Kv]
\geq [K(b)v:Kv]\geq \deg f\;.\]
Hence, equality holds everywhere, and we find that
$[K(b)v:Kv]=[Kv(\zeta):Kv]= [K(b):K]$ and $(vK(b):vK)=1$. That is,
$vK(b)=vK$ and $K(b)v= Kv(\zeta)$. Further, the uniqueness of $v$ on
$K(a)$ follows from the fundamental inequality.

If $fv$ is separable, then so is $f$. Even if $fv$ is not separable but
$v$ is nontrivial on $K$, then $f$ can still be chosen separable since
we can add a summand $cX$ with $c\not=0$, $vc>0$ (we use that $v$ is
non-trivial) without changing the reduction of $f$. On the other hand,
if $fv$ is purely inseparable and hence of the form $X^{p^\nu}-cv$, then
we can choose $f= X^{p^\nu}-c$ which also is purely inseparable if
$\chara K=p$.

Now suppose that $\chara Kv=p>0$ and $\zeta^p\in Kv$, $\zeta\notin Kv$.
Choose $c\in K$ such that $cv=\zeta^p$. To construct an Artin-Schreier
extension, choose any $d\in K$ with $vd<0$, and let $b_0$ be a root of
the Artin-Schreier polynomial $X^p-X-d^p c$. Since $vd^p c=pvd<0$,
Lemma~\ref{vASr} shows that $v(b_0^p-d^p c)=vb_0> vb_0^p$. Consequently,
$v((b_0/d)^p-c)>v(b_0/d)^p=vc=0$, whence $(b_0/d)^pv=cv$ and $(b_0/d)v=
(cv)^{1/p}=\zeta$. We set $b=b_0/d$; so $K(b)=K(b_0)$. As before, it
follows that $vK(b)=vK$ and $K(b)v=Kv(\zeta)$.
\end{proof}

\begin{theorem}                             \label{extprvgrf}
Let $(K,v)$ be an arbitrary valued field. For every extension
$\Gamma|vK$ of ordered abelian groups and every field extension $k|Kv$,
there is an extension $(L,v)$ of $(K,v)$ such that $vL= \Gamma$ and
$Lv=k$. If $\Gamma|vK$ and $k|Kv$ are algebraic, then $L|K$ can be
chosen to be algebraic, with a unique extension of $v$ from $K$ to $L$.
If $\rho=\rr\Gamma/vK$ and $\tau=\trdeg k|Kv$ are finite, then $L|K$ can
be chosen of transcendence degree $\rho+\tau$. If $\Gamma\not=\{0\}$,
then $L$ can be chosen to be a separable extension of $K$.

If both $\Gamma|vK$ and $k|Kv$ are finite, then $L|K$ can be taken to be
a finite extension such that $[L:K]= (\Gamma:vK)[k:Kv]$. If in addition
$v$ is non-trivial on $K$, then $L|K$ can be chosen to be a simple
separable-algebraic extension.

If $\Gamma/vK$ is countable and $k|Kv$ is countably generated, then
$L|K$ can be taken to be a countably generated extension.
\end{theorem}
\begin{proof}
For the proof, we assume that $\Gamma\not=\{0\}$ (the other case is
trivial). Let $\alpha_i\,$, $i\in I$, be a maximal set of elements in
$\Gamma$ rationally independent over $vK$. Then by Lemma~\ref{prelBour}
there is an extension $(K_1,v):=(K(x_i\mid i\in I), v)$ of $(K,v)$ such
that $vK_1=vK\oplus\bigoplus_{i\in I} \Z\alpha_i$ and $K_1 v=Kv$.
Next, choose a transcendence basis $\zeta_j\,$, $j\in J$, of
$k|Kv$. Then by Lemma~\ref{prelBour} there is an extension
$(K_2,v):=(K_1(y_j\mid j\in J),v)$ of $(K_1,v)$ such that $vK_2= vK_1$
and $K_2 v= Kv (\zeta_j\mid j\in J)$. If $\Gamma|vK$ and $k|Kv$ are
algebraic, then $I=J=\emptyset$ and $K_2=K$.

If we have an ascending chain of valued fields whose value groups
are subgroups of $\Gamma$ and whose residue fields are subfields of $k$,
then the union over this chain is again a valued field whose value group
is a subgroup of $\Gamma$ and whose residue field is a subfield of $k$.
So a standard argument using Zorn's Lemma together with the transitivity
of separable extensions shows that there are maximal separable-algebraic
extension fields of $(K_2,v)$ with these properties. Choose one of them
and call it $(L,v)$. We have to show that $vL=\Gamma$ and $Lv=k$. Since
already $\Gamma|vK_2$ and $k|K_2 v$ are algebraic, the same holds for
$\Gamma|vL$ and $k|Lv$. If $vL$ is a proper subgroup of $\Gamma$, then
there is some prime $p$ and some element $\alpha\in\Gamma\setminus vL$
such that $p\alpha\in vL$. But then, Lemma~\ref{constrva} shows that
there exists a proper separable-algebraic extension $(L',v)$ of $(L,v)$
such that $vL'=vL+\Z\alpha \subset \Gamma$ and $L'v= Lv\subset k$, which
contradicts the maximality of $L$. If $Lv$ is a proper subfield of $k$,
then there is some element $\zeta\in k\setminus Lv$ algebraic
over $Lv$. But then, Lemma~\ref{constrra} shows that there exists a
proper separable-algebraic extension $(L',v)$ of $(L,v)$ such that
$vL'=vL\subset \Gamma$ and $L'v= Lv(\zeta)\subset k$, which again
contradicts the maximality of $L$ (here we have used $\Gamma\not=
\{0\}$, which implies that $v$ is not trivial on $L$). This proves that
$vL=\Gamma$ and $Lv=k$, and $(L,v)$ is the required extension of
$(K,v)$. Since $K_2$ is generated over $K$ by a set of elements which
are algebraically independent over $K$, we know that $K_2|K$ is
separable. Since also $L|K_2$ is separable, we find that $L|K$ is
separable. Since $L|K_2$ is algebraic, $\{x_i\,,\,y_j\mid i\in I\,,\,
j\in J\}$ is a transcendence basis of $(K_2|K,v)$. If $\Gamma|vK$ and
$k|Kv$ are algebraic, then $I=J=\emptyset$ and $L$ is an algebraic
extension of $K=K_2\,$.

\pars
If $\Gamma|vK$ and $k|Kv$ are finite, then $L$ can be constructed
by a finite number of applications of Lemma~\ref{constrva} and
Lemma~\ref{constrra}. Since extension degree, ramification index and
inertia degree are multiplicative, we obtain that $[L:K]=(vL:vK)[Lv:Kv]$
$=(\Gamma:vK)[k:Kv]$. If in addition $v$ is non-trivial, then $L|K$ can
be chosen to be a separable extension. Since it is finite, it is simple.

If $\Gamma|vK$ and $k|Kv$ are countably generated algebraic, then
they are unions over a countable chain of algebraic extensions. Hence
also $L$ can be constructed as a union over a countable chain of
algebraic extensions and will thus be a countably generated extension of
$K$.

If $\Gamma|vK$ and $k|Kv$ are countably generated, then the sets
$I$ and $J$ in our above construction are both countable, that is,
$K_2|K$ is countably generated. Moreover, the extensions $\Gamma|vK_2$
and $k|K_2 v$ are countably generated algebraic. Hence by what we have
just shown, $L$ can be taken to be a countably generated extension of
$K_2\,$, and thus also of $K$.
\end{proof}

\pars
Every ordered abelian group is an extension of the trivial group
$\{0\}$ as well as of the ordered abelian group $\Z$. Every field of
characteristic $0$ is an extension of $\Q$, and every field of
characteristic $p>0$ is an extension of $\Fp$. Let $\Gamma\ne 0$ be an
ordered abelian group and $k$ a field. If $\chara k=0$, then $\Q$
endowed with the trivial valuation $v$ will satisfy $v\Q=\{0\}\subset
\Gamma$ and $\Q v=\Q\subset k$. If $\chara k=p>0$, then we can choose
$v$ to be the $p$-adic valuation on $\Q$ to obtain that $v\Q=\Z\subset
\Gamma$ and $\Q v=\Fp\subset k$. But also $\Fp$ endowed with the trivial
valuation $v$ will satisfy $v\Fp=\{0\}\subset \Gamma$ and $\Fp v=
\Fp\subset k$. An application of the foregoing theorem now proves:
\begin{corollary}
Let $\Gamma\ne 0$ be an ordered abelian group and $k$ a field. Then
there is a valued field $(L,v)$ with $vL=\Gamma$ and $Lv=k$. If $\chara
k=p >0$, then $L$ can be chosen to be of characteristic $0$ (mixed
characteristic case) or of characteristic $p$ (equal characteristic
case).
\end{corollary}

For the sake of completeness, we add the following information.
From the fundamental inequality it follows that $v\tilde{K}|vK$,
$vK\sep|vK$, $\tilde{K}v|Kv$ and $K\sep v|Kv$ are algebraic extensions.
On the other hand, Lemma~\ref{constrva} shows that the value group of a
separable-algebraically closed field must be divisible. Similarly, it
follows from Lemma~\ref{constrra} that the residue field of a
separable-algebraically closed non-trivially valued field must be
algebraically closed. This proves:
\begin{lemma}                               \label{valgclo}
Let $(K,v)$ be a non-trivially valued field and extend $v$ to
$\tilde{K}$. Then the value groups $v\tilde{K}$ and $vK\sep$ are equal
to the divisible hull of $vK$, and the residue fields $\tilde{K}v$ and
$K\sep v$ are equal to the algebraic closure of $Kv$.
\end{lemma}

A valued field $(K,v)$ of residue characteristic $p>0$ is called
\bfind{Artin-Schreier closed} if every Artin-Schreier polynomial with
coefficients in $K$ admits a root in $K$. Recall that if $\chara K=p$,
then this means that every Artin-Schreier polynomial with coefficients
in $K$ splits into linear factors over $K$. If $K$ is Artin-Schreier
closed, then so is $Kv$. As a corollary to Lemmas~\ref{constrva}
and~\ref{constrra}, we obtain:
\begin{corollary}                           \label{AScvgrf}
Every Artin-Schreier closed non-trivially valued field of residue
characteristic $p>0$ has $p$-divisible value group and perfect
Artin-Schreier closed residue field.
\end{corollary}

%
%Ä - Ä Ä Ä Ä Ä Ä Ä Ä Ä Ä Ä Ä Ä Ä Ä Ä Ä Ä Ä Ä Ä Ä Ä Ä Ä Ä Ä Ä Ä Ä Ä Ä Ä
%
\subsection{Orderings and valuations}
We will assume the reader to be familiar with the basic theory of
convex valuations, which can be found in [L] and [PR].
\begin{proposition}                         \label{<ext}
Suppose that $(K,<)$ is an ordered field with convex valuation $v$, and
denote by $<_r$ the ordering induced by $<$ through $v$ on $Kv$. Let
$(L|K,v)$ be an extension of valued fields. If $2vL\cap vK=2vK$, then
each extension $<'_r$ of $<_r$ to an ordering of $Lv$ can be lifted
through $v$ to an ordering of $L$ which extends the ordering $<$ of $K$.
\end{proposition}
\begin{proof}
We fix a section from $vK/2vK$ to $K^\times/K^\times{}^2$. Since
$2vL\cap vK=2vK$, this section can be extended to a section from
$vL/2vL$ to $L^\times/L^\times{}^2$. Now there is a bijection between
the set of all liftings of $<_r$ through $v$ to orderings of $K$,
and the set of all group characters of $vK/2vK$; see [PR], (7.5) to
(7.9). The same construction yields a bijection between the set of all
liftings of $<'_r$ through $v$ to orderings of $L$, and the set of all
group characters of $vL/2vL$. Since we use an extension of the section
$vK/2vK\>\rightarrow\>K^\times/K^\times{}^2$ to a section
$vL/2vL\>\rightarrow\>L^\times/L^\times{}^2$, the bijection maps
commute with the restriction from $L$ to $K$ of any lifting. That is,
if a lifting $<'$ of $<'_r$ to $L$ corresponds to a character $\chi$
of $vL/2vL$, then the restriction of $<'$ to $K$ is the unique lifting
of $<_r$ to $K$ which corresponds to the restriction of $\chi$ to
$vK/2vK$.

As the given ordering $<$ of $K$ is a lifting of $<_r\,$, it corresponds
to a unique group character of $vK/2vK$. Since $2vL\cap vK=2vK$, we can
extend it to a group character of $vL/2vL$. Take the lifting of $<'_r$
through $v$ to an ordering $<'$ of $L$ which corresponds to this
group character of $vL/2vL$. Then its restriction to $K$ is $<$.
\end{proof}

%\begin{lemma}                               \label{ordimm}
%Let $K$ be an ordered field with convex valuation $v$, and $(L|K,v)$ an
%immediate extension. Then the ordering of $K$ admits a unique extension
%to an ordering of~$L$ with respect to which $v$ is convex on $L$.
%\end{lemma}
%\begin{proof}
%Recall that $v$ is convex w.r.t.\ an ordering on $L$ if and only if
%the positive cone contains the set $1+{\cal M}_{(L,v)}$ of $1$-units.
%Since $(L|K,v)$ is immediate, for every $a\in L$ there is some $a'\in K$
%such that $v(a-a')>va$. Then we set $\mbox{sign}(a)=\mbox{sign}(a')$.
%This is well defined: if $v(a-a'')>va$, then $v(a'-a'')>va=va'$, hence
%\[\mbox{sign}(a')\>=\>\mbox{sign}(a''(1+(a''-a')/a'))
%\>=\>\mbox{sign}(a'')\]
%since $1+(a''-a')/a'$ is a $1$-unit in $(K,v)$ and $v$ is convex on $K$.
%Now if $a=1+b\in 1+{\cal M}_{(L,v)}\,$, then $v(a-1)>0=va$ and thus,
%$\mbox{sign}(a)=\mbox{sign}(1)$, showing that every $1$-unit is
%positive, which renders $v$ convex.
%
%The extension we have constructed is the only one which
%renders $v$ convex on $L$. Indeed, if $v$ is convex on $L$, then
%$1+{\cal M}_{(L,v)}$ is contained in the positive cone of $L$ and
%therefore, $\mbox{sign}(a)=\mbox{sign}(a(1+(a'-a)/a))
%=\mbox{sign}(a')$ since $1+(a'-a)/a$ is a $1$-unit in $(L,v)$.
%\end{proof}

The following was proved by Knebusch and Wright [KW] and by Prestel
(cf.\ [PR]); see Theorem~5.6 of [L].
\begin{theorem}                             \label{kmwrpr}
Let $v$ be a convex valuation on the ordered field $(K,<)$, $<_r$ the
ordering induced by $<$ on $Kv$, and $R$ a real closure of $(K,<)$. Then
there exists a unique extension of $v$ to a convex valuation of $R$.
Denoting this extension again by $v$, we have that $(R,v)$ is henselian,
$vR$ is divisible, and $Rv$ with the ordering induced by $<$ is a real
closure of $(Kv,<_r)$.
\end{theorem}

\begin{corollary}                           \label{ainrc}
Let $v$ be a convex valuation on the ordered field $(K,<)$
and $R$ a real closure of $(K,<)$, endowed with the unique convex
extension of $v$. Further, let $\Gamma|vK$ be an algebraic extension of
ordered abelian groups, and $k|Kv$ a subextension of some real closure
of $Kv$. Then there is a (separable-algebraic) subextension $(L|K,v)$ of
$(R|K,v)$ such that $vL= \Gamma$ and $Lv=k$. If both $\Gamma|vK$ and
$k|Kv$ are finite, then $L|K$ can be taken to be a finite simple
extension of the form $K(a)|K$ such that $[K(a):K]= (\Gamma:vK)[k:Kv]$.
\end{corollary}

We leave the proof of the corollary as an exercise to the reader. It is
a straightforward application of Hensel's Lemma, using the fact that
$(R,v)$ is henselian. One also uses the fact that all real closures of
$Kv$ are isomorphic over $Kv$, so by passing to an equivalent residue
map (place), one passes from $Rv$ to the real closure given in the
hypothesis.

%
%Ä - Ä Ä Ä Ä Ä Ä Ä Ä Ä Ä Ä Ä Ä Ä Ä Ä Ä Ä Ä Ä Ä Ä Ä Ä Ä Ä Ä Ä Ä Ä Ä Ä Ä
%
\subsection{A version of Krasner's Lemma}
Let $(K,v)$ be any valued field. If $a\in \tilde{K}\setminus K$ is not
purely inseparable over $K$, we choose some extension of $v$
from $K$ to $\tilde{K}$ and define
\[\mbox{\rm kras}(a,K)\;:=\; \max\{v(\tau a-\sigma a)\mid
\sigma,\tau\in \Gal K \mbox{\ \ and\ \ } \tau a\ne \sigma a\}
\;\in\>v\tilde{K}\]
and call it the \bfind{Krasner constant of $a$ over $K$}. Since all
extensions of $v$ from $K$ to $\tilde{K}$ are conjugate, this does not
depend on the choice of the particular extension of $v$. For the
same reason, over a henselian field $(K,v)$ our Krasner constant
$\mbox{\rm kras}(a,K)$ coincides with the Krasner constant
\[\max\{v(a-\sigma a)\mid \sigma\in \Gal K
\mbox{\ \ and\ \ } a\ne \sigma a\}\]
as defined by S.~K.~Khanduja in [KH11,12]. The following is a
variant of the well-known Krasner's Lemma (cf.\ [R]).
\begin{lemma}                        \label{vkras}
Take $K(a)|K$ to be a separable-algebraic extension, and
$(K(a,b),v)$ to be any valued field extension of $(K,v)$ such that
\begin{equation}
v(b-a)\;>\;\mbox{\rm kras}(a,K)\;.
\end{equation}
Then for every extension of $v$ from $K(a,b)$ to its algebraic closure
$\widetilde{K(a,b)}=\widetilde{K(b)}$, the element $a$ lies in the
henselization of $(K(b),v)$ in $(\widetilde{K(b)},v)$.
\end{lemma}
\begin{proof}
Take any extension of $v$ from $K(a,b)$ to $\widetilde{K(b)}$ and denote
by $K(b)^h$ the henselization of $(K(b),v)$ in $(\widetilde{K(b)},v)$.
Since $a$ is separable-algebraic over $K$, it is also
separable-algebraic over $K(b)^h$. Since for every $\rho\in \Gal K(b)^h$
we have that $\rho a=\rho|_{\tilde{K}}a$ and $\rho|_{\tilde{K}}\in \Gal
K$, we find that
\begin{eqnarray*}
\lefteqn{\{v(a-\rho a)\mid\rho\in\Gal K(b)^h \mbox{ and } a\ne\rho a\}}
&&\\
& \subseteq & \{v(a-\sigma a)\mid \sigma\in \Gal K \mbox{ and }
a\ne \sigma a\}\\
& \subseteq & \{v(\tau a-\sigma a)\mid \sigma,\tau\in \Gal K
\mbox{ and } \tau a\ne \sigma a\}\;.
\end{eqnarray*}
This implies that
\[\mbox{\rm kras}(a,K(b)^h)\;\leq\;\mbox{\rm kras}(a,K)\;,\]
and consequently, $v(b-a)\,>\,\mbox{\rm kras}(a,K(b)^h)$. Now
$a\in K(b)^h$ follows from the usual Krasner's Lemma.
\end{proof}

%
%Ä - Ä Ä Ä Ä Ä Ä Ä Ä Ä Ä Ä Ä Ä Ä Ä Ä Ä Ä Ä Ä Ä Ä Ä Ä Ä Ä Ä Ä Ä Ä Ä Ä Ä
%
\section{Valuations on $K(x)$}
%
%
%Ä - Ä Ä Ä Ä Ä Ä Ä Ä Ä Ä Ä Ä Ä Ä Ä Ä Ä Ä Ä Ä Ä Ä Ä Ä Ä Ä Ä Ä Ä Ä Ä Ä Ä
%
\subsection{A basic classification}            \label{sectbc}
In this section, we wish to classify all extensions of the valuation $v$
of $K$ to a valuation of the rational function field $K(x)$. As
\begin{equation}                            \label{111}
1\>=\>\trdeg K(x)|K\>\geq\>\rr vK(x)/vK\,+\,\trdeg K(x)v|Kv
\end{equation}
holds by Lemma~\ref{prelBour}, there are the following mutually
exclusive cases:
\sn
$\bullet$ \ $(K(x)|K,v)$ is \bfind{valuation-algebraic}: \par
$vK(x)/vK$ is a torsion group and $K(x)v|Kv$ is algebraic,
\sn
$\bullet$ \ $(K(x)|K,v)$ is \bfind{value-transcendental}: \par
$vK(x)/vK$ has rational rank 1, but $K(x)v|Kv$ is algebraic,
\sn
$\bullet$ \ $(K(x)|K,v)$ is \bfind{residue-transcendental}: \par
$K(x)v|Kv$ has transcendence degree 1, but $vK(x)/vK$ is a torsion
group.
\mn
We will combine the value-transcendental case and the
residue-transcendental case by saying that
\sn
$\bullet$ \ $(K(x)|K,v)$ is \bfind{valuation-transcendental}: \par
$vK(x)/vK$ has rational rank 1, or $K(x)v|Kv$ has transcendence
degree 1.
\mn
A special case of the valuation-algebraic case is the following:
\sn
$\bullet$ \ $(K(x)|K,v)$ is {\bf immediate}: \par
$vK(x)=vK$ and $K(x)v=Kv$.

\begin{remark}                              \label{limitrem}
It was observed by several authors that a valuation-algebraic extension
of $v$ from $K$ to $K(x)$ can be represented as a limit of an infinite
sequence of residue-transcendental extensions. See, e.g., [APZ3], where
the authors also derive the assertion of our Theorem~\ref{count} from
this fact. A ``higher form'' of this approach is found in [S]. The
approach is particularly important because residue-transcendental
extensions behave better than valuation-algebraic extensions: the
corresponding extensions of value group and residue field are finitely
generated (Corollary~\ref{fingentb}), and they do not generate a defect:
see the Generalized Stability Theorem (Theorem~3.1) and its application
in [KKU1].
\end{remark}

\parb
If $K$ is algebraically closed, then the residue field $Kv$ is
algebraically closed, and the value group $vK$ is divisible. So we see
that for an extension $(K(x)|K,v)$ with algebraically closed $K$, there
are only the following mutually exclusive cases:
\sn
$(K(x)|K,v)$ is {\bf immediate}: \ $vK(x)=vK$ and $K(x)v=Kv$,\n
$(K(x)|K,v)$ is \bfind{value-transcendental}: \ $\rr vK(x)/vK=1$,
but $K(x)v=Kv$,\n
$(K(x)|K,v)$ is \bfind{residue-transcendental}: \ $\trdeg K(x)v|Kv=1$,
but $vK(x)=vK$.

\parm
Let us fix an arbitrary extension of $v$ to $\tilde{K}$. Every valuation
$w$ on $K(x)$ can be extended to a valuation on $\tilde{K}(x)$. If $v$
and $w$ agree on $K$, then this extension can be chosen in such a way
that its restriction to $\tilde{K}$ coincides with $v$. Indeed, if $w'$
is any extension of $w$ to $\tilde{K}(x)$ and $v'$ is its restriction to
$\tilde{K}$, then
%
%by Lemma~\ref{conj}, Theorem~\ref{algprolconj},
%
there is an automorphism $\tau$ of $\tilde{K}|K$ such that $v'\tau=v$ on
$\tilde{K}$. We choose $\sigma$ to be the (unique) automorphism of
$\tilde{K}(x)|K(x)$ whose restriction to $\tilde{K}$ is $\tau$ and which
satisfies $\sigma x=x$. Then $w'\sigma$ is an extension of $w$ from
$K(x)$ to $\tilde{K}(x)$ whose restriction to $\tilde{K}$ is
$v$. We conclude:
\begin{lemma}                               \label{restrall}
Take any extension of $v$ from $K$ to its algebraic closure
$\tilde{K}$. Then every extension of $v$ from $K$ to $K(x)$ is the
restriction of some extension of $v$ from $\tilde{K}$ to
$\tilde{K}(x)$.
\end{lemma}

Now extend $v$ to $\widetilde{K(x)}$.
%By Lemma~\ref{conj},
We know that $v\widetilde{K(x)}/vK(x)$ and
$v\tilde{K}/vK$ are torsion groups, and also
$v\tilde{K}(x)/vK(x)\subset v\widetilde{K(x)}/vK(x)$
is a torsion group. Hence,
\[\rr v\tilde{K}(x)/v\tilde{K}\>=\>\rr vK(x)/vK\;.\]
Since $v\tilde{K}$ is divisible,
%by Lemma~\ref{conj}. Hence,
$vK(x)/vK$ is a torsion group if and only if
$v\tilde{K}(x) = v\tilde{K}$.

%Again by Lemma~\ref{conj},
Further, the extensions $\widetilde{K(x)}v|K(x)v$ and
$\tilde{K}v|Kv$ are algebraic, and also the subextension
$\tilde{K}(x)v|K(x)v$ of $\widetilde{K(x)}v|K(x)v$ is
algebraic. Hence,
\[\trdeg \tilde{K}(x)v|\tilde{K}v\>=\>
\trdeg K(x)v|Kv\;.\]
Since $\tilde{K}v$ is algebraically closed,
$K(x)v|Kv$ is algebraic if and only if
$\tilde{K}(x)v = \tilde{K}v$. We have proved:
\begin{lemma}                               \label{aavavtrt}
$(K(x)|K,v)$ is valuation-algebraic if and only if
$(\tilde{K}(x)|\tilde{K},v)$ is immediate.
$(K(x)|K,v)$ is valuation-transcendental if and only if
$(\tilde{K}(x)|\tilde{K},v)$ is not immediate.
$(K(x)|K,v)$ is value-transcendental if and only if
$(\tilde{K}(x)|\tilde{K},v)$ is value-transcendental.
$(K(x)|K,v)$ is residue-transcendental if and only if
$(\tilde{K}(x)|\tilde{K},v)$ is residue-transcendental.
\end{lemma}

The proof can easily be generalized to show:
\begin{lemma}                               \label{a-i}
Let $(F|K,v)$ be any valued field extension. Then $vF|vK$ and
$Fv|Kv$ are algebraic if and only if $(F.\tilde{K}|\tilde{K},v)$ is
immediate, for some (or any) extension of $v$ from $F$ to $F.\tilde{K}$.
\end{lemma}

%
%Ä - Ä Ä Ä Ä Ä Ä Ä Ä Ä Ä Ä Ä Ä Ä Ä Ä Ä Ä Ä Ä Ä Ä Ä Ä Ä Ä Ä Ä Ä Ä Ä Ä Ä
%
\subsection{Pure and weakly pure extensions}    \label{sectpure}
Take $t\in K(x)$. If $vt$ is not a torsion element modulo $vK$, then
$t$ will be called a \bfind{value-transcendental element}. If $vt=0$
and $tv$ is transcendental over $Kv$, then $t$ will be called a
\bfind{residue-transcendental element}. An element will be called a
\bfind{valuation-transcendental element} if it is value-transcendental
or residue-transcendental. We will call the extension $(K(x)|K,v)$
\bfind{pure} ({\bf in} $x$) if one of the following cases holds:
\sn
$\bullet$ \ for some $c,d\in K$, $d\cdot (x-c)$ is
valuation-transcendental,
\sn
$\bullet$ \ $x$ is the pseudo limit of some pseudo Cauchy sequence
in $(K,v)$ of transcendental type.
\sn
We leave it as an exercise to the reader to prove that $(K(x)|K,v)$ is
pure in $x$ if and only if it is pure in any other generator of $K(x)$
over $K$; we will not need this fact in the present paper.

If $(K(x)|K,v)$ is pure in $x$ then it follows from Lemma~\ref{prelBour}
and Lemma~\ref{limtpcs} that $x$ is transcendental over $K$. If $d\cdot
(x-c)$ is value-transcendental, then $vK(x)=vK\oplus\Z v(x-c)$ and
$K(x)v=Kv$ by Lemma~\ref{prelBour} (in this case, we may in fact choose
$d=1$). If $d\cdot (x-c)$ is residue-transcendental, then again by
Lemma~\ref{prelBour}, we have $vK(x)=vK$ and that $K(x)v=
Kv(d(x-c)v)$ is a rational function field over $Kv$. If $x$ is the
pseudo limit of some pseudo Cauchy sequence in $(K,v)$ of transcendental
type, then $(K(x)|K,v)$ is immediate by Lemma~\ref{limtpcs}. This
proves:
\begin{lemma}                               \label{purevgrf}
If $(K(x)|K,v)$ is pure, then $vK$ is pure in $vK(x)$ (i.e., $vK(x)/vK$
is torsion free), and $Kv$ is relatively algebraically closed in $K(x)v$.
\end{lemma}

\parm
Here is the ``prototype'' of pure extensions:
\begin{lemma}                               \label{acpure}
If $K$ is algebraically closed and $x\notin K$, then $(K(x)|K,v)$ is
pure.
\end{lemma}
\begin{proof}
Suppose that the set
\begin{equation}                            \label{v(x-K)}
v(x-K)\;:=\;\{v(x-b)\mid b\in K\}
\end{equation}
has no maximum. Then there is a pseudo Cauchy sequence in $(K,v)$
with pseudo limit $x$, but without a pseudo limit in $K$. Since $K$ is
algebraically closed, Theorem~3 of [KA] shows that this pseudo Cauchy
sequence must be of transcendental type. The extension therefore
satisfies the third condition for being pure.

\pars
Now assume that the set $v(x-K)$ has a maximum, say, $v(x-c)$ with $c\in
K$. If $v(x-c)$ is a torsion element over $vK$, then $v(x-c)\in vK$
because $vK$ is divisible as $K$ is algebraically closed. It then
follows that there is some $d\in K$ such that $vd(x-c)=0$. If $d(x-c)v$
were algebraic over $Kv$, then it were in $Kv$ since $K$ and thus
also $Kv$ is algebraically closed. But then, there were some $b_0\in
K$ such that $v(d(x-c)-b_0)>0$. Putting $b:=c+d^{-1}b_0\,$, we would
then obtain that $v(x-b)=v((x-c)-d^{-1}b_0)>-vd=v(x-c)$, a contradiction
to the maximality of $v(x-c)$. So we see that either $v(x-c)$ is
non-torsion over $vK$, or there is some $d\in K$ such $vd(x-c)=0$ and
$d(x-c)v$ is transcendental over $Kv$. In both cases, this shows that
$(K(x)|K,v)$ is pure.
\end{proof}

\pars
We will call the extension $(K(x)|K,v)$ \bfind{weakly pure} ({\bf in}
$x$) if it is pure in $x$ or if there are $c,d\in K$ and $e\in\N$ such
that $vd(x-c)^e=0$ and $d(x-c)^ev$ is transcendental over $Kv$.

\begin{lemma}                               \label{icpure}
Assume that the extension $(K(x)|K,v)$ is weakly pure. If we take any
extension of $v$ to $\widetilde{K(x)}$ and take $K^h$ to be the
henselization of $K$ in $(\widetilde{K(x)},v)$, then $K^h$ is the
implicit constant field of this extension:
\[K^h\;=\;\ic (K(x)|K,v)\;.\]
\end{lemma}
\begin{proof}
As noted already in the introduction, $K^h$ is contained in $\ic
(K(x)|K,v)$. Since $K(x)^h$ is the fixed field of the decomposition
group $G\dec_x:= G\dec(K(x)\sep|K(x),v)$ in the separable-algebraic
closure $K(x)\sep$ of $K(x)$, we know that $\ic (K(x)|K,v)$ is the fixed
field in $K\sep$ of the subgroup
\[G_{\rm res}\;:=\;\{\sigma|_{K\sep}\mid \sigma\in G\dec_x\}\]
of $\Gal K$. In order to show our assertion, it suffices to
show that $\ic (K(x)|K,v)\subseteq K^h$, that is, that the decomposition
group $G\dec:=G\dec(K\sep|K,v)$ is contained in $G_{\rm res}\,$. So we
have to show: if $\tau$ is an automorphism of $K\sep|K$ such that
$v\tau=v$ on $K\sep$, then $\tau$ can be lifted to an automorphism
$\sigma$ of $K(x)\sep|K(x)$ such that $v\sigma=v$ on $K(x)\sep$. In
fact, it suffices to show that $\tau$ can be lifted to an automorphism
$\sigma$ of $K\sep(x)|K(x)$ such that $v\sigma=v$ on $K\sep(x)$. Indeed,
then we take any extension $\sigma'$ of $\sigma$ from $K\sep (x)$ to
$K(x)\sep$. Since the extensions $v\sigma'$ and $v$ of $v$ from $K\sep
(x)$ to $K(x)\sep$ are conjugate, there is some $\rho\in\Gal
(K(x)\sep|K\sep (x))$ such that $v\sigma'\rho=v$ on $K(x)\sep$. Thus,
$\sigma:=\sigma'\rho\in G\dec$ is the desired lifting of $\tau$ to
$K(x)\sep$.

We take $\sigma$ on $K\sep (x)$ to be the unique automorphism which
satisfies $\sigma x=x$ and $\sigma|_{K\sep}=\tau$. Using that $(K(x)
|K,v)$ is weakly pure, we have to show that $v\sigma=v$ on $K\sep(x)$.
Assume first that for some $c,d\in K$ and $e\in\N$, $d(x-c)^e$ is
valuation-transcendental. Since $K(x-c)=K(x)$, we may assume w.l.o.g.\
that $c=0$. Every element of $K\sep(x)$ can be written as a quotient
of polynomials in $x$ with coefficients from $K\sep$. For every
polynomial $f(x)=a_nx^n+\ldots+a_1x+a_0\in K\sep[x]$,
\begin{eqnarray*}
v\sigma f(x) & = & v\left(\sigma(a_n)x^n+ \ldots+\sigma(a_1)x+
\sigma(a_0)\right)\\
 & = & \min_i (v\sigma(a_i)+ivx) \;=\; \min_i (v\tau(a_i)+ivx)\\
 & = & \min_i (va_i+ivx) \;=\; vf(x)\;,
\end{eqnarray*}
where the second equality holds by Lemma~\ref{prelBour} and
Lemma~\ref{algtransc}. This shows that $v\sigma=v$ on $K\sep(x)$.

Now assume that $x$ is the pseudo limit of a pseudo Cauchy sequence in
$(K,v)$ of transcendental type. By Lemma~\ref{persist}, this pseudo
Cauchy sequence is also of transcendental type over $(K\sep,v)$. Observe
that $x$ is still a pseudo limit of this pseudo Cauchy sequence in
$(K\sep(x),v\sigma)$, because $v\sigma(x-a)=v(\sigma x-\sigma a)=
v(x-a)$ for all $a\in K$. But $v\sigma=v\tau=v$ on $K\sep$, and the
extension of $v$ from $K\sep$ to $K\sep(x)$ is uniquely determined by
the pseudo Cauchy sequence (cf.\ Theorem~\ref{Ka2}). Consequently,
$v\sigma=v$ on $K\sep(x)$.
\end{proof}

%
%ÄÄÄÄÄÄÄÄÄÄÄÄÄÄÄÄÄÄÄÄÄÄÄÄÄÄÄÄÄÄÄÄÄÄÄÄÄÄÄÄÄÄÄÄÄÄÄÄÄÄÄÄÄÄÄÄÄÄÄÄÄÄÄÄÄÄÄÄÄÄÄ
%
\subsection{Construction of nasty examples}          \label{sectne}
We are now able to give the
\sn
{\bf Proof of Theorem~\ref{piltant}:}
\sn
Let $K$ be any algebraically closed field of characteristic $p>0$. On
$K(x)$, we take $v$ to be the $x$-adic valuation. We work in the power
series field $K((\frac{1}{p^\infty}\Z))$ of all power series in $x$ with
exponents in $\frac{1}{p^\infty}\Z$, the $p$-divisible hull of $\Z$. We
choose $y$ to be a power series
\begin{equation}                            \label{yps}
y\>=\>\sum_{i=1}^{\infty} x^{-p^{-e_i}}
\end{equation}
where $e_i$ is any increasing sequence of natural numbers such that
$e_{i+1}\geq e_i+i$ for all $i$. We then restrict the canonical
($x$-adic) valuation of $K((\frac{1}{p^\infty}\Z))$ to $K(x,y)$ and call
it again $v$. We show first that $vK(x,y)=\frac{1}{p^\infty}\Z$. Indeed,
taking $p^{e_j}$-th powers and using that the characteristic of $K$ is
$p$, we find
\[y^{p^{e_j}}-\sum_{i=1}^{j} x^{-p^{e_j-e_i}}\>=\>
\sum_{i=j+1}^{\infty} x^{-p^{e_j-e_i}}\;.\]
Since $e_j-e_i\geq 0$ for $i\leq j$, the left hand side is an element in
$K(x,y)$. The right hand side has value
\[-p^{e_j-e_{j+1}} vx\;;\]
since $e_j-e_{j+1}\leq -j$, we see that $\frac{1}{p^j}vx$ lies in
$vK(x,y)$. Hence, $\frac{1}{p^\infty}\Z\subseteq vK(x,y)$. On the other
hand, $vK(x,y)\subseteq vK((\frac{1}{p^\infty}\Z))=\frac{1}{p^\infty}\Z$
and therefore, $vK(x,y)=\frac{1}{p^\infty}\Z$.

By definition, $y$ is a pseudo limit of the pseudo Cauchy sequence
\[\left(\sum_{i=1}^{\ell} x^{-p^{-e_i}}\right)_{\ell\in\N}\]
in the field $L=K(x^{1/p^i}\mid i\in\N)\subset
K((\frac{1}{p^\infty}\Z))$.
Suppose it were of algebraic type. Then by [KA], Theorem~3, there would
exist some algebraic extension $(L(b)|L,v)$ with $b$ a pseudo limit of
the sequence. But then $b$ would also be algebraic over $K(x)$ and hence
the extension $K(x,b)|K(x)$ would be finite. On the other hand, since
$b$ is a pseudo limit of the above pseudo Cauchy sequence, it can be
shown as before that $vK(x,b)=\frac{1}{p^\infty}\Z$ and thus,
$(vK(x,b):vK(x))= \infty$. This contradiction to the fundamental
inequality shows that the sequence must be of transcendental type. Hence
by Lemma~\ref{icpure}, $L^h$ is relatively algebraically closed in
$L(y)^h$. Since $L^h=L.K(x)^h$ is a purely inseparable extension of
$K(x)^h$ and $K(x,y)^h|K(x)^h$ is separable, this shows that $K(x)^h$ is
relatively algebraically closed in $K(x,y)^h$.

Now we set $\eta_0:=\frac{1}{x}$, and by induction on $i$ we choose
$\eta_i\in \widetilde{K(x)}$ such that $\eta_i^p-\eta_i=\eta_{i-1}\,$.
Then we have
\[v\eta_i\>=\> -\frac{1}{p^i} vx\]
for every $i$. Since $vK(x)^h=vK(x)=\Z vx$,
this shows that $K(x)^h(\eta_i)|K(x)^h$
has ramification index at least $p^i$. On the other hand, it has degree
at most $p^i$ and therefore, it must have degree and ramification index
equal to $p^i$. Note that for all $i\geq 0$, $K(x,\eta_i)=K(x,\eta_1,
\ldots,\eta_i)$ and every extension $K(x,\eta_{i+1})|K(x,\eta_i)$
is a Galois extension of degree and ramification index $p$. By what we
have shown, the chain of these extensions is linearly disjoint from
$K(x)^h|K(x)$. Since $K(x)^h$ is relatively algebraically closed in
$K(x,y)^h$ and the extensions are separable, it follows that the chain
is also linearly disjoint from $K(x,y)^h|K(x)$.

We will now show that all extensions $K(x,y,\eta_i)|K(x,y)$ are
immediate. First, we note that $K(x,y)v=K$ since $K\subset K(x,y)
\subset K((\frac{1}{p^\infty}\Z))$ and $K((\frac{1}{p^\infty}\Z))v=K$.
Since $K$ is algebraically closed, the inertia degree of the extensions
must be $1$. Further, as the ramification index of a Galois extension is
always a divisor of the extension degree, the ramification index of
these extensions must be a power of $p$. But the value group of $K(x,y)$
is $p$-divisible, which yields that the ramification index of the
extensions is $1$. By what we have proved above, they are linearly
disjoint from $K(x,y)^h|K(x,y)$, that is, the extension of the
valuation is unique. This shows that the defect of each extension
$K(x,y,\eta_i)|K(x,y)$ is equal to its degree $p^i$.           \QED

\sn
\begin{remark}
Instead of defining $y$ as in (\ref{yps}), we could also
use any power series
\begin{equation}                            \label{yps1}
y\>=\>\sum_{i=1}^{\infty} x^{n_i p^{-e_i}}
\end{equation}
where $n_i\in \Z$ are prime to $p$ and the sequence $n_i p^{-e_i}$ is
strictly increasing. The example in [CP] is of this form. But in this
example, the field $K(x,y)$ is an extension of degree $p^2$ of a field
$K(u,v)$ such that the extension of the valuation from $K(u,v)$ to
$K(x,y)$ is unique. Since the value group of $K(x,y)$ is
$\frac{1}{p^\infty}\Z$, it must be equal to that of $K(u,v)$. Since $K$
is algebraically closed, both have the same residue field. Therefore,
the extension has defect $p^2$. This shows that we can also use
subfields instead of field extensions to produce defect extensions, in
quite the same way.

A special case of (\ref{yps1}) is the power series
\begin{equation}                            \label{yps2}
y\>=\>\sum_{i=1}^{\infty} x^{i-p^{-e_i}}
\>=\>\sum_{i=1}^{\infty} x^i x^{-p^{-e_i}}
\end{equation}
which now has a support that is cofinal in $\frac{1}{p^\infty}\Z$.
\end{remark}

\parm
To conclude this section, we use Lemma~\ref{icpure} to construct an
example about relatively closed subfields in henselian fields. The
following fact is well known:
\n
{\it Suppose that $K$ is relatively closed in a henselian valued field
$(L,v)$ of residue characteristic $0$ and that $Lv|Kv$ is algebraic.
Then $vL/vK$ is torsion free.}
\n
We show that the assumption ``$Lv|Kv$ is algebraic'' is necessary.
\sn
\begin{example}
On the rational function field $\Q(x)$, we take $v$ to
be the $x$-adic valuation. We extend $v$ to the rational function field
$\Q(x,y)$ in such a way that $vy=0$ and $yv$ is transcendental over
$\Q(x)v=\Q$. So by Lemma~\ref{prelBour} we have $v\Q(x,y)=v\Q(x)=\Z vx$
and $\Q(x,y)v=\Q (yv)$. We pick any integer $n>1$. Then $v\Q(x^n)=\Z
nvx$ and $\Q(x^n)v=\Q$. Further, $v\Q(x^n,xy)=\Z vx$ since $vx=vxy\in
v\Q(x^n,xy)\subseteq \Z vx$. Also, $\Q(x^n,xy)v=\Q(y^nv)$ by
Lemma~\ref{algtransc}. From Lemma~\ref{icpure} we infer that $\Q(x^n)^h$
is relatively algebraically closed in $\Q(x^n,xy)^h$. But
\[v\Q(x^n,xy)^h/v\Q(x^n)^h \>=\> v\Q(x^n,xy)/v\Q(x^n) \>=\> \Z vx/
\Z nvx \>\isom\> \Z/n\Z\]
is a non-trivial torsion group.
\end{example}

%
%ÄÄÄÄÄÄÄÄÄÄÄÄÄÄÄÄÄÄÄÄÄÄÄÄÄÄÄÄÄÄÄÄÄÄÄÄÄÄÄÄÄÄÄÄÄÄÄÄÄÄÄÄÄÄÄÄÄÄÄÄÄÄÄÄÄÄÄÄÄÄÄ
%
\subsection{All valuations on $K(x)$}            \label{sectae}
In this section, we will explicitly define extensions of a given
valuation on $K$ to a valuation on $K(x)$. First, we define
valuation-transcendental extensions, using the idea of valuation
independence. Let $(K,v)$ be an arbitrary valued field, and $x$
transcendental over $K$. Take $a\in K$ and an element $\gamma$ in some
ordered abelian group extension of $vK$. We define a map
$v_{a,\gamma}:
\, K(x)^{\times} \rightarrow vK+ \Z\gamma$\glossary{$v_{a,\gamma}$} as
follows. Given any $g(x)\in K[x]$ of degree $n$, we can write
\begin{equation}                            \label{g(x)dev}
g(x)=\sum_{i=0}^{n} c_i (x-a)^i\;.
\end{equation}
Then we set
\begin{equation}                            \label{defvag}
v_{a,\gamma}\,g(x)\>:=\>\min_{0\leq i\leq n} vc_i + i\gamma\;.
\end{equation}
We extend $v_{a,\gamma}$ to $K(x)$ by setting
$v_{a,\gamma}(g/h):=v_{a,\gamma}g-v_{a,\gamma} h$.

\pars
For example, the valuation $v_{0,0}$ is called \bfind{Gau{\ss}
valuation} or
\bfind{functional valuation} and is given by
\[v_{0,0}\,(c_nx^n+\ldots+c_1x+c_0)\>=\>\min_{0\leq i\leq n}
vc_i\;.\]

\begin{lemma}                               \label{vag}
$v_{a,\gamma}$ is a valuation which extends $v$ from $K$ to $K(x)$.
It satisfies:\sn
1) If $\gamma$ is non-torsion over $vK$, then $v_{a,\gamma}K(x)=vK
\oplus\Z\gamma$ and $K(x)v_{a,\gamma}=Kv$.
\sn
2) If $\gamma$ is torsion over $vK$, $e$ is the smallest positive
integer such that $e\gamma\in vK$ and $d\in K$ is some element such that
$vd=-e\gamma$, then $d(x-a)^ev_{a,\gamma}$ is transcendental over $Kv$,
$K(x)v_{a,\gamma}=Kv(d(x-a)^ev_{a,\gamma})$ and $v_{a,\gamma}K(x)=
vK+\Z\gamma$. In particular, if $\gamma=0$ then $(x-a)v_{a,\gamma}$ is
transcendental over $Kv$, $K(x)v_{a,\gamma}=Kv\,((x-a)v_{a,\gamma})$ and
$v_{a,\gamma}K(x)=vK$.
\end{lemma}
\begin{proof}
It is a straightforward exercise to prove that $v_{a,\gamma}$ is a
valuation and that 1) and 2) hold. However, one can also deduce this
from Lemma~\ref{prelBour}. It says that if we assign a non-torsion value
$\gamma$ to $x-a$ then we obtain a unique valuation which satisfies
(\ref{defvag}). Since this defines a unique map $v_{a,\gamma}$ on
$K[x]$, we see that $v_{a,\gamma}$ must coincide with the valuation
given by Lemma~\ref{prelBour}, which in turn satisfies assertion 1).
Similarly, if $\gamma\in vK$, $d\in K$ with $vd=-\gamma$ and we assign a
transcendental residue to $d(x-a)$, then Lemma~\ref{prelBour} gives us a
valuation on $K(x)$ which satisfies (\ref{defvag}) and hence must
coincide with $v_{a,\gamma}\,$. This shows that $v_{a,\gamma}$ is a
valuation and satisfies 2).

If $e>1$, then we can first use Lemma~\ref{prelBour} to see that
$v_{a,\gamma}$ is a valuation on the subfield $K(d(x-a)^e)$ of $K(x)$
and that $v_{a,\gamma}K(d(x-a)^e)=vK$ and $K(d(x-a)^e)
v_{a,\gamma}= Kv(d(x-a)^ev_{a,\gamma})$ with $d(x-a)^ev_{a,\gamma}$
transcendental over $Kv$. We know that there is an extension $w$ of
$v_{a,\gamma}$ to $K(x)$. It must satisfy $w(x-a)=-vd/e=\gamma$. So
$0,w(x-a)$, $w(x-a)^2, \ldots,w(x-a)^{e-1}$ lie in distinct cosets
modulo $vK$. From Lemma~\ref{algvind} it follows that $w$ satisfies
(\ref{defvag}) on $K(x)$, hence it must coincide with the valuation
$v_{a,\gamma}$ on $K(x)$. Assertion 2) for this case follows from
Lemma~\ref{algvind}.
\end{proof}

Now we are able to prove:

\begin{theorem}                             \label{allext}
Take any valued field $(K,v)$. Then all extensions of $v$ to the
rational function field $\tilde{K}(x)$ are of the form
\sn
$\bullet$ \ $\tilde{v}_{a,\gamma}$ where $a\in \tilde{K}$ and $\gamma$
is an element of some ordered group extension of $vK$, or
\n
$\bullet$ \ $\tilde{v}_{\bA}$ where \bA\ is a pseudo Cauchy
sequence in $(\tilde{K},\tilde{v})$ of transcendental type,
\sn
where $\tilde{v}$ runs through all extensions of $v$ to $\tilde{K}$.
The extension is of the form $\tilde{v}_{a,\gamma}$ with $\gamma\notin
\tilde{v}\tilde{K}$ if and only if it is value-transcendental, and
with $\gamma\in\tilde{v}\tilde{K}$ if and only if it is
residue-transcendental. The extension is of the form $\tilde{v}_{\bA}$
if and only if it is valuation-algebraic.

\pars
All extensions of $v$ to $K(x)$ are obtained by restricting the
above extensions, already from just one fixed extension $\tilde{v}$ of
$v$ to $\tilde{K}$.
\end{theorem}
\begin{proof}
By Lemma~\ref{vag} and Theorem~\ref{Ka2}, $\tilde{v}_{a,\gamma}$ and
$\tilde{v}_{\bA}$ are extensions of $\tilde{v}$ to $\tilde{K}(x)$. For
the converse, let $w$ be any extension of $v$ to $\tilde{K}(x)$ and set
$\tilde{v}=w|_{\tilde{K}}\,$. From Lemma~\ref{acpure} we know that
$(\tilde{K}(x)|\tilde{K},w)$ is always pure. Hence,
%
%for every extension $w$ of $\tilde{v}$ to $K(x)$,
%
either $d(x-c)$ is valuation-transcendental for some $c,d\in
\tilde{K}$, or $x$ is the pseudo limit of some pseudo Cauchy sequence
\bA\ in $(\tilde{K},\tilde{v})$ of transcendental type. In the first
case, Lemma~\ref{prelBour} shows that
\[w \sum_{i=0}^{n} d_i (d (x-c))^i\>=\>\min_{0\leq i\leq n}
vd_i+iwd (x-c) \>=\>\min_{0\leq i\leq n} vd_i d^i + iw(x-c)\]
for all $d_i\in\tilde{K}$. This shows that $w=\tilde{v}_{c,\gamma}$
for $\gamma=w(x-c)$. If this value is not in $\tilde{v}\tilde{K}$, then
it is non-torsion over $\tilde{v}\tilde{K}$ and thus, the extension of
$\tilde{v}$ to $\tilde{K}(x)$, and hence also the extension of $v$ to
$K(x)$, is value-transcendental. If it is in $\tilde{v}\tilde{K}$, then
the residue of $d (x-c)$ is not in $\tilde{K}\tilde{v}$, and the
extension of $\tilde{v}$ to $\tilde{K}(x)$, and hence also the extension
of $v$ to $K(x)$, is residue-transcendental.

In the second case, we know from Theorem~\ref{Ka2} that \bA\ induces
an extension $\tilde{v}_{\bA}$ of $\tilde{v}$ to $\tilde{K}(x)$ such
that $x$ is a pseudo limit of \bA\ in $(\tilde{K}(x),\tilde{v}_{\bA})$.
Since $x$ is also a pseudo limit of \bA\ in $(\tilde{K}(x),w)$, we can
infer from Lemma~\ref{limtpcs} that $w=\tilde{v}_{\bA}\,$. It also
follows from Theorem~\ref{Ka2} that $(\tilde{K}(x)|\tilde{K},
\tilde{v}_{\bA})$ is immediate and consequently, $(\tilde{K}(x)|K,
\tilde{v}_{\bA})$ is valuation-algebraic.

\pars
For the last assertion, we invoke Lemma~\ref{restrall}. Now it just
remains to show that it suffices to take the restrictions of the
valuations $\tilde{v}_{a,\gamma}$ and $\tilde{v}_{\bA}$ for one fixed
$\tilde{v}$. Suppose that $\tilde{w}$ is another extension of $v$ to
$\tilde{K}$. Since all such extensions are conjugate, there is $\sigma
\in \Gal K$ such that $\tilde{w}=\tilde{v}\circ\sigma$. Let $g(x)\in
K[x]$ be given as in (\ref{g(x)dev}). Extend $\sigma$ to an
automorphism of $\tilde{K}(x)$ which satisfies $\sigma x=x$. Since
$g$ has coefficients in $K$, we then have
\[g(x)\>=\>\sigma g(x)\>=\>\sum_{i}\sigma c_i (x-\sigma a)^i\]
and therefore,
\[\tilde{w}_{a,\gamma}\,g(x)\>=\>\min_i (\tilde{w}c_i +i\gamma)
\>=\>\min_i (\tilde{v}\sigma c_i +i\gamma) \>=\>
\tilde{v}_{\sigma a,\gamma}\,g(x)\;.\]
This shows that $\tilde{w}_{a,\gamma}=\tilde{v}_{\sigma a,\gamma}$
on $K(x)$.

Given a pseudo Cauchy sequence \bA\ in $(\tilde{K},\tilde{w})$, we set
$\bA_{\sigma}=(\sigma a_{\nu})_{\nu<\lambda}\,$. This is a
pseudo Cauchy sequence in $(\tilde{K},\tilde{v})$ since $\tilde{v}
(\sigma a_{\mu}-\sigma a_{\nu})=\tilde{v}\sigma (a_{\mu}-
a_{\nu})=\tilde{w} (a_{\mu}- a_{\nu})$. For every polynomial
$f(x)\in \tilde{K}[x]$, we have $\tilde{v} f(\sigma a_{\nu})=
\tilde{w}\sigma^{-1} (f(\sigma a_{\nu}))=\tilde{w}(\sigma^{-1}
(f))(a_{\nu})$, where $\sigma^{-1}(f)$ denotes the polynomial obtained
from $f(x)$ by applying $\sigma^{-1}$ to the coefficients. So we see
that $\bA_{\sigma}$ is of transcendental type if and only if \bA\ is. If
$g(x)\in K[x]$, then $\sigma^{-1}(g)=g$ and the above computation shows
that $\tilde{v} g(\sigma a_{\nu})=\tilde{w} g(a_{\nu})$. This implies
that $\tilde{w}_{\bA}=\tilde{v}_{\bA_{\sigma}}$ on $K(x)$.
\end{proof}

\begin{remark}                              \label{allextrem}
If $v$ is trivial on $K$, hence $Kv=K$ (modulo an isomorphism), and if
we choose $\gamma>0$, then the restriction $w$ of $\tilde{v}_{a,\gamma}$
to $K(x)$ will satisfy $xw=aw=a$. It follows that $K(x)w=K(a)$. Further,
$wK(x)\subseteq \Z\gamma$ and thus, $wK(x)\isom\Z$.
\end{remark}

%
%Ä - Ä Ä Ä Ä Ä Ä Ä Ä Ä Ä Ä Ä Ä Ä Ä Ä Ä Ä Ä Ä Ä Ä Ä Ä Ä Ä Ä Ä Ä Ä Ä Ä Ä
%
\subsection{Prescribed implicit constant fields}       \label{sectMT}
If not stated otherwise, we will always assume that $(K,v)$ is any
valued field. The following is an immediate consequence of our version
of Krasner's Lemma:
\begin{lemma}                               \label{lvx-a}
Assume that $K(a)|K$ is a separable-algebraic extension. Assume further
that $K(x)|K$ is a rational function field and $v$ is a valuation on
$\widetilde{K(x)}$ such that
\begin{equation}                            \label{vx-a}
v(x-a)\> > \> \mbox{\rm kras}(a,K)\;.
\end{equation}
Then $K(a)\subseteq (K(x)|K,v)$, and consequently,
\[vK(a)\subseteq vK(x)\;\mbox{\ \ and\ \ }\; K(a)v\subseteq K(x)v\;.\]
\end{lemma}
\begin{proof}
By Lemma~\ref{vkras}, condition (\ref{vx-a}) implies that $K(a)\subseteq
K(x)^h$, the henselization being chosen in $(\widetilde{K(x)},v)$.
Consequently, $K(a)\subseteq \ic (K(x)|K,v)$, $vK(a)\subseteq
vK(x)^h=vK(x)$ and $K(a)v\subseteq K(x)^hv=K(x)v$.
\end{proof}

\begin{proposition}                         \label{finhcf}
Assume that $(K(a)|K,v)$ is a separable-algebraic extension of valued
fields. Further, take $\Gamma$ to be the abelian group $vK(a)\oplus\Z$
endowed with any extension of the ordering of $vK(a)$, and take $k$
to be the rational function field in one variable over $K(a)v$. Then
there exists an extension $v_1$ of $v$ from $K(a)$ to
$\widetilde{K(x)}$ such that
\begin{equation}                        \label{v1}
v_1K(x)\>=\>\Gamma\mbox{\ \ and\ \ }K(x)v_1\>=\>K(a)v
\mbox{\ \ and\ \ } K(a)^h\>=\>\ic(K(x)|K,v_1)\;.
\end{equation}
If $v$ is non-trivial on $K$, then there exists an extension $v_2$ of
$v$ from $K(a)$ to $\widetilde{K(x)}$ such that
\begin{equation}                        \label{v2}
v_2K(x)\>=\>vK(a)\mbox{\ \ and\ \ }K(x)v_2\>=\>k
\mbox{\ \ and\ \ } K(a)^h\>=\>\ic(K(x)|K,v_2)\;.
\end{equation}
If $(K(a),v)$ admits a transcendental immediate extension, then there
is also an extension $v_3$ such that
\begin{equation}                                  \label{v3}
v_3K(x)=vK(a)\mbox{\ \ and\ \ }K(x)v_3=K(a)v\;.
\end{equation}
If in addition $(K(a),v)$ admits a pseudo Cauchy sequence of
transcendental type, then $v_3$ can be chosen such that
$K(a)^h=\ic(K(x)|K,v_3)$.

Suppose that $(K(a),<)$ is an ordered field and that $v$ is convex.
Denote by $<_r$ the ordering induced by $<$ on the residue field
$K(a)v$. Then $<_r$ can be lifted through $v_1$ and through $v_3$ to
$K(x)$ in such a way that the lifted orderings extend $<$ (from $K$). If
$<'_r$ is an extension of $<_r$ to $k$, then $<'_r$ can be lifted
through $v_2$ to $K(x)$ in such a way that the lifted ordering extends
$<$ (from $K$).
\end{proposition}
\begin{proof}
If $K(a)=K$, then we do the following. We take a generator $\gamma$ of
$\Gamma$ over $vK$ and let $v_1$ be any extension of $v_{0,\gamma}$ to
$\widetilde{K(x)}$. Further, we take $v_2$ to be any extension of
$v_{0,0}$ to $\widetilde{K(x)}$. Then the desired properties of $v_1$
and $v_2$ follow from Lemma~\ref{vag} and Lemma~\ref{icpure}. To
construct $v_3$, we send $x$ to some transcendental element in the given
immediate extension of $(K,v)$. The so obtained embedding induces an
extension $v_3$ of $v$ to $K(x)$ such that $(K(x)|K,v_3)$ is immediate.
We extend $v_3$ further to $\widetilde{K(x)}$. If $(K,v)$ admits a
pseudo Cauchy sequence of transcendental type, then we can use
Theorem~\ref{Ka2} to construct an immediate extension $v_3$ of $v$ to
$K(x)$ such that $x$ is a pseudo limit of this pseudo Cauchy sequence.
Then by Lemma~\ref{icpure}, $K^h\>=\>\ic(K(x)|K,v_3)$.

\pars
Now assume that $a\notin K$.
If $v'$ is any extension of $v$ to $\widetilde{K(x)}$ such that $v'(x-a)
>\mbox{\rm kras}(a,K)$, then Lemma~\ref{lvx-a} shows that $a\in K(x)^h$,
where the henselization is taken in $(\widetilde{K(x)},v')$.
Thus,
\begin{equation}                            \label{inhcf}
K(x)^h\;=\;K(a,x)^h\;,
\end{equation}
and consequently,
\begin{equation}                            \label{v'vgrf}
\left.\begin{array}{lcl}
v'K(x) & = & v'K(x)^h\;=\;v'K(a,x)^h\;=\;v'K(a,x)\\
K(x)v' & = & K(x)^hv'\;=\;K(a,x)^hv'\;=\;K(a,x)v'\;.
\end{array}\right\}
\end{equation}
If in addition we know that
\begin{equation}                            \label{icKa}
K(a)^h\>=\>\ic(K(a,x)|K(a),v')\;,
\end{equation}
then
\[K(a)^h\>\subseteq\>\ic(K(x)|K,v')\>\subseteq\>\ic(K(a,x)|K(a),v')\>=\>
K(a)^h\;,\]
which yields that $K(a)^h=\ic(K(x)|K,v')$.

\parm
As $vK$ is cofinal in its divisible hull $v\tilde{K}$, we can choose
some $\alpha\in vK$ such that $\alpha\geq\mbox{\rm kras}(a,K)$.

\pars
To construct $v_1\,$, we choose any positive generator $\beta$ of
$\Gamma$ over $vK(a)$. Then also $\gamma:=\alpha+\beta$ is a generator
of $\Gamma$ over $vK(a)$. Now we take $v_1$ to be any extension of
$v_{a,\gamma}$ to $\widetilde{K(x)}$. From Lemma~\ref{vag} we know that
$v_{a,\gamma}K(a,x)=vK(a)\oplus\Z\gamma=\Gamma$ and
$K(a,x)v_{a,\gamma} =K(a)v\,$. Since
\begin{equation}                            \label{vgx-a}
v_{a,\gamma}(x-a)\;=\;\gamma\;>\; \mbox{\rm kras}(a,K)\;,
\end{equation}
equations (\ref{v'vgrf}) hold for $v'=v_{a,\gamma}$, and we obtain:
\begin{eqnarray*}
v_1K(x) & = & v_{a,\gamma}K(a,x)\;=\;vK(a)+\Z\gamma\;=\;\Gamma\\
K(x)v_1 & = & K(a,x)v_{a,\gamma}\;=\;K(a)v\;.
\end{eqnarray*}
From Lemma~\ref{icpure} we infer that (\ref{icKa})
holds for $v'=v_{a,\gamma}$. Consequently, $K(a)^h=\ic(K(x)|K,v_1)$.

\pars
To construct $v_2\,$, we make use of our assumption that $v$ is
non-trivial on $K$ and choose some positive $\beta\in vK$. We set
$\gamma:=\alpha+\beta$ and take $v_2$ to be any extension of
$v_{a,\gamma}$ to $\widetilde{K(x)}$. From Lemma~\ref{vag} we know
that $v_{a,\gamma}K(a,x)=vK(a)=\Gamma$ and $K(a,x)v_{a,\gamma}$
is a rational function field in one variable over $K(a)v$; in this
sense, $K(a,x)v_{a,\gamma}=k\,$. Again we have (\ref{vgx-a}), and
equations (\ref{v'vgrf}) hold for $v'=v_{a,\gamma}$; so we obtain
\begin{eqnarray*}
v_2K(x) & = & v_{a,\gamma}K(a,x)\;=\;vK(a)\\
K(x)v_2 & = & K(a,x)v_{a,\gamma}\;=\;k\;.
\end{eqnarray*}
From Lemma~\ref{icpure} we infer that (\ref{icKa}) holds for
$v'=v_{a,\gamma}$. Consequently, $K(a)^h=\ic(K(x)|K,v_2)$.

\pars
To construct $v_3\,$, we take $(L|K(a),v)$ to be the transcendental
immediate extension which exists by hypothesis. We choose some $y\in L$,
transcendental over $K(a)$ and such that $\gamma:=vy>\mbox{\rm kras}
(a,K)$. Then $(K(a,y)|K(a),v)$ is immediate. Now the isomorphism
$K(a,y)\isom K(a,x)$ induced by $y\mapsto x-a$ induces on $K(a,x)$
a valuation $v'$ such that $(K(a,x)|K(a),v')$ is immediate and
$v'(x-a)=\gamma$. We take $v_3$ to be any extension of $v'$ to
$\widetilde{K(x)}$. Again, we have (\ref{v'vgrf}), and we obtain
\begin{eqnarray*}
v_3K(x) & = & v'K(a,x)\;=\;vK(a)\\
K(x)v_3 & = & K(a,x)v'\;=\;K(a)v\;.
\end{eqnarray*}
If $(K(a),v)$ admits a pseudo Cauchy sequence of transcendental type,
then we can use Theorem~\ref{Ka2} to construct an immediate extension
$(K(a,z)|K(a),v)$. Multiplying $z$ with a suitable element from $K$
will give us $y$ such that $vy>\mbox{\rm kras}(a,K)$. By our above
construction, also $x$ will be a pseudo limit of a pseudo Cauchy
sequence of transcendental type. Then by Lemma~\ref{icpure},
$K(a)^h\>=\>\ic(K(x)|K,v_3)$.

\parm
Finally, suppose that $(K(a),<)$ is an ordered field and that $<_r$ is
the ordering induced by $<$ on $K(a)v$. We have that $vK(a)=v_2 K(a,x)
=v_3 K(a,x)$, so it is trivially true that $2v_2 K(a,x)\cap vK(a)=
2vK(a)=2v_3 K(a,x)\cap vK(a)$. Further, $v_1 K(a,x)=vK(a)\oplus
\Z\gamma$, hence also in this case, $2v_1 K(a,x)\cap vK(a)=2vK(a)$.
Therefore, Proposition~\ref{<ext} shows that the given orderings
on $K(a)v$ and $k$ can be lifted through $v_1\,$, $v_2\,$ and
$v_3$ respectively, to orderings on $K(a,x)$ which extend the
ordering $<$ of $K(a)$.
%Hence, their restrictions to $K(x)$ are extensions of the ordering $<$
%of $K$.
\end{proof}

\pars
This proposition proves Theorem~\ref{MT} in the case where $K_1^h|K^h$
is finite since then, there is some $a\in K_1$ such that $K_1^h=K(a)^h$.
Further, we obtain:

\begin{proposition}                             \label{5.4}
Take any finite ordered abelian group extension $\Gamma_0$ of $vK$ and
any finite field extension $k_0$ of $Kv$. Further, take $\Gamma$ to be
the abelian group $\Gamma_0 \oplus\Z$ endowed with any extension of the
ordering of $\Gamma_0\,$, and take $k$ to be the rational function field
in one variable over $k_0\,$. If $v$ is trivial on $K$, then assume in
addition that $k_0|Kv$ is separable. Then there exists an extension
$v_1$ of $v$ from $K$ to $K(x)$ such that
\begin{equation}                                   \label{vv1}
v_1K(x)\>=\>\Gamma\mbox{\ \ and\ \ }K(x)v_1\>=\>k_0\;.
\end{equation}
If $v$ is non-trivial on $K$, then there exists an extension $v_2$ of
$v$ from $K$ to $K(x)$ such that
\begin{equation}                                   \label{vv2}
v_2K(x)\>=\>\Gamma_0\mbox{\ \ and\ \ }K(x)v_2\>=\>k\;.
\end{equation}
If $(K,v)$ admits a transcendental immediate extension, then
there is also an extension $v_3$ such that
\begin{equation}                                  \label{vv3}
v_3K(x)=\Gamma_0\mbox{\ \ and\ \ }K(x)v_3=k_0\;.
\end{equation}

Suppose that $(K,<)$ is an ordered field and that $v$ is convex.
Suppose further that $k_0$ and $k$ are equipped with extensions of the
ordering induced by $<$ on $Kv$. Then these extensions can be lifted
through $v_1\,$, $v_2\,$, $v_3$ to $K(x)$ in such a way that the lifted
orderings extend $<$.
\end{proposition}
\begin{proof}
We choose any finite separable extension $(K(a)|K,v)$ such that $vK(a)=
\Gamma_0\,$, $K(a)v=k_0\,$ and $[K(a):K]=(vK(a):vK)[K(a)v:Kv]$. If $v$
is non-trivial on $K$, then such an extension exists by
Theorem~\ref{extprvgrf}; otherwise, $K(a)$ is just equal to $k_0\,$, up
to the isomorphism induced by the residue map of the trivial valuation
$v$. If $(K,v)$ admits a transcendental immediate extension, then by
Lemma~\ref{persimm}, also $(K(a),v)$ admits a transcendental immediate
extension. Now the first part of our theorem follows from
Proposition~\ref{finhcf}.

Suppose that $(K,<)$ is an ordered field with $v$ convex. Then by
Corollary~\ref{ainrc} we can choose the field $K(a)$ to be a subfield of
a real closure $(R,<)$ of $K$, equipped with a convex extension of $v$,
in such a way that the given ordering on $k_0$ is induced by $<$ through
this extension of $v$. Now again, our assertion for the ordered case
follows from Proposition~\ref{finhcf}.
\end{proof}

\parm
We turn to the realization of countably infinite
separable-algebraic extensions as implicit constant fields.
\begin{proposition}                         \label{inf-i}
Let $(K_1|K,v)$ be a countably infinite separable-algebraic extension of
non-trivially valued henselian fields. Then $(K_1,v)$ admits a pseudo
Cauchy sequence of transcendental type. In particular, there is an
extension $v_3$ of $v$ to $\widetilde{K(x)}$ such that $(K_1(x)|K_1,v)$
is immediate, with $x$ being the pseudo limit of this pseudo Cauchy sequence.
Moreover, $K_1=\ic(K(x)|K,v_3)$.
\end{proposition}
\begin{proof}
$K_1|K$ is a countably infinite union of finite subextensions. Thus,
we can choose a sequence $a_i\,$, $i\in\N$ such that $K(a_i)|K$ is
separable-algebraic and $K(a_i)\subsetuneq K(a_{i+1})$ for all $i$, and
such that $K_1=\bigcup_{i\in\N} K(a_i)$. Through multiplication with
elements from $K$ it is possible to choose each $a_i$ in such a way that
\begin{equation}                            \label{vai+1>k}
va_{i+1}\;>\; \max\{va_i,\mbox{\rm kras}(a_i,K(a_1,\ldots,a_{i-1}))\}\;.
\end{equation}
We set
\[b_i\;:=\;\sum_{j=1}^{i} a_j\;.\]
By construction, $(b_i)_{i\in\N}$ is a pseudo Cauchy sequence.
Suppose that $x$ is a pseudo limit of it, for some
extension of $v$ from $K_1$ to $K_1(x)$. Then by induction on $i$,
we show that $a_i\in K(x)^h$, where $K(x)^h$ is chosen in some
henselization of $(K_1(x),v)$. First, by Lemma~\ref{vkras}, $v(x-a_1)
=v(x-b_1)=va_2> \mbox{\rm kras}(a_1,K)$ implies that $a_1\in K(x)^h$. If
we have already shown that $a_1,\ldots,a_i\in K(x)^h$, then $b_i\in
K(x)^h$ and $v(x-b_i-a_{i+1})=v(x-b_{i+1})=va_{i+2}> \mbox{\rm kras}
(a_{i+1},K(a_1,\ldots,a_i))$ implies that $a_{i+1}\in
K(a_1,\ldots,a_i)(x-b_i)^h=K(a_1,\ldots,a_i)(x)^h=K(x)^h$.

This proves that $K_1\subseteq K(x)^h$. Since $K_1|K$ is infinite, this
also proves that $x$ must be transcendental. As a pseudo Cauchy sequence
of algebraic type would admit an algebraic pseudo limit ([KA],
Theorem~3), this yields that $(b_i)_{i\in\N}$ is a pseudo Cauchy
sequence of transcendental type in $(K_1,v)$.

By Theorem~\ref{Ka2} we can now extend $v$ to $K_1(x)$ such
that $(K_1(x)|K_1,v)$ is immediate and $x$ is a pseudo limit of
$(b_i)_{i\in\N}$. We choose any extension $v_3$ of $v$ from $K_1(x)$
to $\widetilde{K(x)}$. By Lemma~\ref{icpure} we know that $K_1=
\ic(K_1(x)|K_1,v_3)$. Hence,
\[K_1\>\subseteq\>K(x)^h\>\subseteq\>\ic(K(x)|K,v_3)\>\subseteq\>
\ic(K_1(x)|K_1,v_3) \>=\> K_1\;,\]
which shows that $K_1=\ic(K(x)|K,v_3)$.
\end{proof}

\pars
Since $\ic(K(x)|K,v)=\ic(K^h(x)|K^h,v)$, this proposition proves
Theorem~\ref{MT} in the case where $K_1^h|K^h$ is countably infinite.

\begin{proposition}                             \label{cgcg}
Take any non-trivially valued field $(K,v)$, any countably generated
ordered abelian group extension $\Gamma_0$ of $vK$ such that
$\Gamma_0/vK$ is a torsion group, and any countably generated algebraic
field extension $k_0$ of $Kv$. Assume that $\Gamma_0/vK$ or $k_0|Kv$ is
infinite. Then there exists an extension $v_3$ of $v$ from $K$ to $K(x)$
such that equations (\ref{vv3}) hold.

Suppose that $(K,<)$ is an ordered field and that $v$ is convex. If
$k_0$ is equipped with an extension of the ordering induced by $<$ on
$Kv$, then this can be lifted through $v_3$ to an ordering of $K(x)$
which extends $<$.
\end{proposition}
\begin{proof}
We fix an extension of $v$ to $\tilde{K}$. By Theorem~\ref{extprvgrf}
there is a countably generated separable-algebraic extension $K_1|K^h$
such that $vK_1=\Gamma_0$ and $K_1v=k_0\,$. Since at least one of the
extensions $\Gamma_0|vK$ and $k_0|Kv$ is infinite, $K_1^h|K^h$ is
infinite too, taking the henselizations in $(\tilde{K},v)$. Hence by
Proposition~\ref{inf-i} there is an extension $v_3$ of $v$ to $K_1(x)$
such that $K_1\subset K(x)^h$ and that $(K_1(x)|K_1,v_3)$ is immediate.
Consequently, $K_1(x)^h=K(x)^h$, which gives us $v_3K(x)= v_3K(x)^h=
v_3K_1(x)=vK_1=\Gamma_0$ and $K(x)v_3=K(x)^hv_3=K_1(x)v_3=K_1v =k_0\,$.

Suppose that $(K,<)$ is ordered, $v$ is convex and $k_0$ is equipped
with an extension of the ordering induced by $<$ on $Kv$. Take any
real closure $(R,<)$ of $(K,<)$ with an extension of $v$ to a convex
valuation of $R$. By Corollary~\ref{ainrc} we can choose the extension
$K_1|K$ as a subextension of $R|K$, with the ordering on $k_0$ induced
by $<$ through $v$. Since $(K_1(x)|K_1, v_3)$ is immediate,
Proposition~\ref{<ext} shows that there is a lifting of the ordering of
$k_0$ through $v$ to an ordering of $K_1(x)$ which extends the ordering
$<$ of $K_1\,$. Its restriction to $K(x)$ is an extension of the
ordering $<$ of $K$.
\end{proof}

\parm
To conclude this section, we now give an alternative
\sn
{\bf Proof of Theorem~\ref{piltant}:}
\sn
Let $K$ be an algebraically closed field of characteristic $p>0$. On
$K(x)$, we again take $v$
to be the $x$-adic valuation. We assume that $K$ contains an element $t$
which is transcendental over the prime field of $K$. Then it can be
proved that $K(x)^h$ admits two infinite linearly disjoint towers of
Galois extensions of degree $p$ and ramification index $p$. They can be
defined as follows. For the first tower, we set $\eta_0=x^{-1}$, take
$\eta_{i+1}$ to be a root of $X^p-X-\eta_i\,$, and set $L_i:=
K(\eta_i)$. For the second tower, we set $\vartheta_0=tx^{-1}$,
take $\vartheta_{i+1}$ to be a root of $X^p-X-t\vartheta_i\,$, and set
$N_i:=K(\vartheta_i)$. Note that $L_0=N_0=K(x)$, $\eta_i\in L_{i+1}$
and $\vartheta_i\in N_{i+1}\,$. For each $i\geq 0$, we have that
$v\eta_i=v\vartheta_i=-\frac{1}{p^i}vx$, and the extensions
$L_{i+1}|L_i$ and $N_{i+1}|N_i$ are Galois of degree $p$ with
ramification index $p$. Consequently, the same is true for the
extensions $L_{i+1}^h|L_i^h$ and $N_{i+1}^h|N_i^h$ (note that
$L_i^h=L_i.K(x)^h$ and $N_i^h=N_i.K(x)^h$).

We set $L:=\bigcup_{i\in\N}L_i$ and $N:=\bigcup_{i\in\N}N_i\,$. By the
above, $L$ and $N$ are linearly disjoint from $K(x)^h$ over $K(x)$.
Thus, $L^h=L.K(x)^h$ and $N^h=N.K(x)^h$ are countably infinite
separable-algebraic extensions of $K(x)^h$, and it can be proved that
they are linearly disjoint over $K(x)^h$. We use Proposition~\ref{inf-i}
(with $K$ replaced by $K(x)^h$ and $x$ replaced by $y$) to obtain an
extension of $v$ from $K(x)^h$ to $K(x)^h (y)$ such that $vK(x)^h(y)=
vL^h= \frac{1}{p^\infty}\Z$ and $K(x)^h(y)v= L^hv= K(x)^hv=K$, and such
that the extension $(K(x)^h(y)|K(x)^h,v)$ has implicit constant field
$L^h$ (i.e., $L^h$ is relatively algebraically closed in
$K(x,y)^h= (K(x)^h(y))^h\,$). Since
%$K(x,y)|K(x)$ is separable and
$K(x)$ is relatively algebraically closed in $K(x,y)$, we see that
$N$ is linearly disjoint from $K(x,y)$ over $K(x)$ and therefore,
$N.K(x,y)| K(x,y)$ is again an infinite tower of Galois extensions of
degree $p$. Since $N$ is linearly disjoint from $K(x)^h$ over $K(x)$
and $N^h=N.K(x)^h$ is linearly disjoint from $L^h$ over $K(x)^h$, we see
that $N$ is linearly disjoint from $L^h$ over $K(x)$. Since
%$K(x,y)^h|L^h$ is separable and
$L^h$ is relatively algebraically closed in $K(x,y)^h$, this implies
that $N$ is linearly disjoint from $K(x,y)^h$ over $K(x)$ and hence,
$N.K(x,y)$ is linearly disjoint from $K(x,y)^h$ over $K(x,y)$.
Therefore, the extension of $v$ from $K(x,y)$ to $N.K(x,y)$ is unique.
Since $(K(x)^h(y))^h=K(x,y)^h$, we see that $(K(x)^h(y)|K(x,y),v)$ is
immediate. So we have
$vK(x,y)=\frac{1}{p^\infty}\Z$, which is $p$-divisible, and
$K(x,y)v=K$, which is algebraically closed. Hence, the extension
$(N.K(x,y)|K(x,y),v)$ is immediate, and so it is an infinite tower of
Galois extensions of degree $p$ and defect $p$.            \QED

\mn
\begin{remark}
For the above defined $\eta_i$, we have that the roots of
$X^p-X-\eta_{i-1}$ are $\eta_i,\eta_i+1,\ldots,\eta_i+p-1$. Therefore,
\[\mbox{\rm kras}(\eta_i,K(\eta_{i-1}))\>=\>0\]
for all $i$. Setting $a_i\>:=\> x^i\eta_i$, we obtain that
\begin{eqnarray*}
\mbox{\rm kras}(a_i,K(x)^h(a_1,\ldots,a_{i-1})) & = &
\mbox{\rm kras}(a_i,K(x)^h(\eta_{i-1})) \\
& = & ivx+\mbox{\rm kras}(\eta_i,K(x)^h(\eta_{i-1})) \>=\>ivx
\end{eqnarray*}
and $va_{i+1}=(i+1)vx-\frac{1}{p^{i+1}}vx>ivx$. This shows that
$va_{i+1} >va_i$ and that (\ref{vai+1>k}) is satisfied. So we can take
$y$ to be a pseudo limit of the Cauchy sequence $(\sum_{j=1}^{i}
a_j)_{i\in\N}\,$. That is,
\begin{equation}
y\>=\>\sum_{i=1}^{\infty} x^i \eta_i\;.
\end{equation}
Comparing this with (\ref{yps2}), we see that we have replaced the term
$x^{-p^{-e_i}}$ by $\eta_i$ which has an infinite expansion in powers of
$x$, starting with $x^{-p^{-i}}$.
\end{remark}

%
%Ä - Ä Ä Ä Ä Ä Ä Ä Ä Ä Ä Ä Ä Ä Ä Ä Ä Ä Ä Ä Ä Ä Ä Ä Ä Ä Ä Ä Ä Ä Ä Ä Ä Ä
%
\section{Rational function fields of higher transcendence
degree}                                     \label{secthtd}
This section is devoted to the proof of Theorems~\ref{MTsevvar}
and~\ref{extord}. We will make use of the following theorem which we
prove in [KU5]:
\begin{theorem}                             \label{MTai}
Let $(L|K,v)$ be a valued field extension of finite transcendence
degree $\geq 0$, with $v$ non-trivial on $L$. Assume that one of the
following two cases holds:\sn
\underline{transcendental case}: \
$vL/vK$ has rational rank at least 1 or $Lv|Kv$ is transcendental;
\sn
\underline{separable-algebraic case}: \
$L|K$ contains a separable-algebraic subextension $L_0|K$
such that within some henselization of $L$, the corresponding
extension $L_0^h|K^h$ is infinite.
\sn
Then each maximal immediate extension of $(L,v)$ has infinite
transcendence degree over~$L$.
\end{theorem}
\n
Note: The assertion need not be true for an infinite purely inseparable
extension $L|K$.

\pars
Now we can give the
\pars
{\bf Proof of Theorem~\ref{MTsevvar}}: \
Assume that $(K,v)$ is a valued field, $n,\rho,\tau,\ell$ are
non-negative integers, $\Gamma$ is an ordered abelian group extension of
$vK$ such that $\Gamma/vK$ is of rational rank $\rho$, and $k|Kv$ is a
field extension of transcendence degree $\tau$. We pick a maximal set of
elements $\alpha_1,\ldots,\alpha_{\rho}$ in $\Gamma$ rationally
independent over $vK$, and a transcendence basis $\zeta_1,\ldots,
\zeta_{\tau}$ of $k|Kv$.

\pars
We prove {\bf Part A} of Theorem~\ref{MTsevvar} first. So assume further
that $n>\rho+\tau$, $\Gamma|vK$ and $k|Kv$ are countably generated, and
at least one of them is infinite. By Lemma~\ref{prelBour} there is a
unique extension of $v$ from $K$ to $K(x_1,\ldots,x_{\rho+\tau})$ such
that $vx_i=\alpha_i$ for $1\leq i \leq\rho$, and $x_{\rho+i}v=\zeta_i$
for $1\leq i\leq\tau$. Since $\Gamma|vK$ and $k|Kv$ are countably
generated, $vK(x_1,\ldots, x_{\rho+\tau})$ contains $\alpha_1,\ldots,
\alpha_{\rho}$ and $K(x_1,\ldots,x_{\rho+\tau})v$ contains $\zeta_1,
\ldots,\zeta_{\tau}$, the extensions $\Gamma|vK(x_1,\ldots,
x_{\rho+\tau})$ and $k|K(x_1,\ldots,x_{\rho+\tau})v$ are countably
generated and algebraic.

Suppose first that at least one of these extensions is infinite.
From our assumption that $\Gamma\ne\{0\}$ it follows that $v$ is
non-trivial on $K(x_1,\ldots,x_{\rho+\tau})$. Hence we can use
Proposition~\ref{cgcg} to find an extension of $v$ from
$K(x_1,\ldots,x_{\rho+\tau})$ to $K(x_1,\ldots,x_{\rho+\tau+1})$ such
that $vK(x_1,\ldots,x_{\rho+ \tau+1})=\Gamma$ and $K(x_1,\ldots,
x_{\rho+\tau+1})v=k$, and that the implicit constant field of
$(K(x_1,\ldots,x_{\rho+\tau+1})|K(x_1,\ldots,x_{\rho+\tau}),v)$ is an
infinite separable-algebraic extension of $K(x_1,\ldots,
x_{\rho+\tau})^h$. That means that $K(x_1,\ldots,x_{\rho+\tau+1})^h|
K(x_1,\ldots,x_{\rho+\tau})^h$ contains an infinite separable-algebraic
sub\-extension. Hence by the separable-algebraic case of
Theorem~\ref{MTai}, each maximal immediate extension of
$K(x_1,\ldots,x_{\rho+\tau+1})^h$ is of infinite transcendence degree.
This shows that we can find an immediate extension of $v$ from
$K(x_1,\ldots,x_{\rho+\tau+1})^h$ to $K(x_1,\ldots,x_{\rho+\tau+1})^h
(x_{\rho+\tau+2},\ldots,x_n)$. Its restriction to $K(x_1,\ldots,x_n)$ is
an immediate extension of $v$ from $K(x_1,\ldots,x_{\rho+\tau+1})$.

\pars
Now suppose that both extensions $\Gamma|vK(x_1,\ldots, x_{\rho+\tau})$
and $k|K(x_1,\ldots,x_{\rho+\tau})v$ are finite. Then from our
hypothesis that at least one of the extensions $\Gamma|vK$ and $k|Kv$ is
infinite, it follows that $\rho>0$ or $\tau>0$. Hence by the
transcendental case of Theorem~\ref{MTai}, any maximal immediate
extension of $(K(x_1,\ldots,x_{\rho+\tau}),v)$ is of infinite
trans\-cendence degree. In combination with Proposition~\ref{5.4},
this shows that there is an extension of $v$ from
$K(x_1,\ldots,x_{\rho+\tau})$ to $K(x_1,\ldots,x_{\rho+\tau+1})$
such that $vK(x_1,\ldots,x_{\rho+ \tau+1})=\Gamma$ and
$K(x_1,\ldots,x_{\rho+\tau+1})v=k$. By the transcendental case of
Theorem~\ref{MTai}, every maximal immediate extension of
$(K(x_1,\ldots,x_{\rho+\tau+1}),v)$ is of infinite transcendence degree.
So we can find an extension of $v$ to $K(x_1,\ldots,x_n)$ as in the
previous case.

\pars
In both cases we have that $vK(x_1,\ldots,x_n)=\Gamma$ and
$K(x_1,\ldots,x_n)v=k$, as required.

\parb
Now we prove {\bf Part B}. So assume that $n\geq\rho+\tau$, and that
$\Gamma|vK$ and $k|Kv$ are finitely generated.

{\bf I)} \ First, we consider the case of $\rho>0$. By
Lemma~\ref{prelBour} there is an extension of $v$ from $K$ to
$K(x_1,\ldots,x_{\rho-1+\tau})$ such that $vx_i=\alpha_i$ for $1\leq i
\leq\rho-1$, and $x_{\rho-1+i}v= \zeta_i$ for $1\leq i\leq\tau$. Since
$\Gamma|vK$ and $k|Kv$ are finitely generated, $vK(x_1,\ldots,
x_{\rho-1+\tau})$ contains $\alpha_1,\ldots,\alpha_{\rho-1}$ and
$K(x_1,\ldots,x_{\rho-1+\tau})v$ contains $\zeta_1,\ldots,\zeta_{\tau}$,
the extension $\Gamma| vK(x_1,\ldots,x_{\rho-1+\tau})$ is finitely
generated of rational rank 1 (and thus, $\Gamma$ is of the form
$\Gamma_0\oplus\Z$ with $\Gamma_0/vK(x_1,\ldots,x_{\rho-1+\tau})$
finite), and the extension $k|K(x_1, \ldots, x_{\rho-1+\tau})v$ is
finite.

If $v$ is not trivial on $K(x_1, \ldots,x_{\rho-1+\tau})$, then
we can apply Proposition~\ref{5.4} to obtain an extension of $v$
from $K(x_1,\ldots,x_{\rho-1+\tau})$ to $K(x_1,\ldots,x_{\rho+\tau})$
such that $vK(x_1,\ldots,x_{\rho+ \tau})=\Gamma$ and
$K(x_1,\ldots,x_{\rho+\tau})v =k$. From the transcendental case of
Theorem~\ref{MTai} we see that any maximal immediate extension of
$K(x_1,\ldots,x_{\rho+\tau}),v)$ is of infinite transcendence degree.
This shows that we can find an immediate extension of $v$ from
$K(x_1,\ldots,x_{\rho+\tau})$ to $K(x_1,\ldots,x_n)$.

Now suppose that $v$ is trivial on $K(x_1,\ldots,x_{\rho-1+\tau})$. Then
$\rho=1$ and $v$ is trivial on $K$. It follows that $\Gamma\isom\Z$, and
we pick a generator $\gamma$ of $\Gamma$.

If $n>1+\tau$, we modify the above procedure in such a way that we
use Lemma~\ref{prelBour} to choose an extension of $v$ from
$K(x_1,\ldots,x_{\tau})$ to $K(x_1,\ldots,x_{1+\tau})$ so that
the latter has value group $\Z\gamma=\Gamma$ and residue field
$Kv(\zeta_1,\ldots,\zeta_{\tau})$. Then we use Proposition~\ref{5.4} in
combination with the transcendental case of Theorem~\ref{MTai} to find
an extension of $v$ to $K(x_1,\ldots,x_n)$ with value group $\Gamma$
and residue field $k$.

If $n=1+\tau$, then by our assumption B2), $k$ is a simple algebraic
extension of a rational function field $k'$ in $\tau$ variables over
$Kv$ (or of $Kv$ itself if $\tau=0$), or a rational function field in
one variable over a finitely generated field extension $k_0$ of $Kv$ of
transcendence degree $\tau-1$.

Assume the first case holds. The extension of $v$ to the rational
function field $K(x_1,\ldots,x_{\tau})$ can be chosen such that
$k'=Kv(x_1v,\ldots,x_{\tau}v)= K(x_1,\ldots, x_{\tau})v$. Then the
extension $k|K(x_1, \ldots, x_{\tau})v$ is simple algebraic, and
according to Remark~\ref{allextrem} there is an
extension of $v$ to the rational function field
$K(x_1,\ldots,x_{\tau})(x_{1+\tau})$ which satisfies
$vK(x_1,\ldots,x_{1+\tau})=\Z\gamma=\Gamma$ and
$K(x_1,\ldots,x_{1+\tau})v =k$.

In the second case, we know that $\tau\geq 1$. The elements $\zeta_i\in
k$ can be chosen in such a way that $\zeta_1,\ldots,\zeta_{\tau-1}$ form
a transcendence basis of $k_0|Kv$. We pick $\gamma\in\Gamma$, $\gamma\ne
0$. By Lemma~\ref{prelBour} there is an extension of $v$ from $K$ to
$K(x_1,\ldots,x_{\tau})$ such that $x_{i}v=\zeta_i$ for $1\leq
i\leq\tau-1$, and $vx_{\tau}=\gamma$. Since $K(x_1,\ldots,x_{\tau})v
=Kv(\zeta_1,\ldots, \zeta_{\tau-1})$, we see that $k$ is a rational
function field in one variable over a finite extension of
$K(x_1,\ldots,x_{\tau})v$. Further, $vK(x_1,\ldots,x_{\tau})
=\Gamma$. Now we use Proposition~\ref{5.4} to find an extension of $v$
to $K(x_1,\ldots,x_{1+\tau})$ with value group $\Gamma$ and residue
field $k$.

\pars
{\bf II)} \ Second, suppose that $\rho=0$. Then because of $\Gamma\ne
\{0\}$, we know that $v$ is non-trivial on $K$. If $\tau>0$, then we
proceed as follows. The case of $n>\tau$ is covered by Part A. So we
assume that $n=\tau$, and that $k$ is a rational function field in one
variable over a finitely generated field extension $k_0$ of $Kv$ of
transcendence degree $\tau-1$. Again we choose
$\zeta_1,\ldots,\zeta_{\tau-1}$ to form a transcendence basis of
$k_0|Kv$, and use Lemma~\ref{prelBour} to find an extension of $v$ from
$K$ to $K(x_1,\ldots,x_{\tau-1})$ such that $x_{i}v=\zeta_i$ for
$1\leq i\leq\tau-1$. Again we obtain that $k$ is a rational function
field in one variable over a finite extension of
$K(x_1,\ldots,x_{\tau-1})v$. By assumption,
$\Gamma/vK(x_1,\ldots,x_{\tau-1})=\Gamma/vK$ is finite. Hence by
Proposition~\ref{5.4} there is an extension of $v$ from
$K(x_1,\ldots,x_{\tau-1})$ to $K(x_1,\ldots,x_{\tau})$ such that
$vK(x_1,\ldots,x_{\tau})=\Gamma$ and $K(x_1,\ldots,x_{\tau})v=k$.

Finally, suppose that $\rho=0=\tau$ and that there is an immediate
extension $(K',v)$ of $(K,v)$ which is either infinite
separable-algebraic or of transcendence degree at least $n$. If the
former holds, then we obtain from the separable-algebraic case of
Theorem~\ref{MTai} that every maximal immediate extension of $(K',v)$
has infinite transcendence degree. But a maximal immediate extension of
$(K',v)$ is also a maximal immediate extension of $(K,v)$. Thus in all
cases, we have the existence of immediate extensions of $(K,v)$ of
transcendence degree at least $n$. Therefore, we can choose an immediate
extension of $v$ from $K$ to $K(x_1,\ldots, x_{n-1})$ such that
$(K(x_1,\ldots, x_{n-1}),v)$ still admits a transcendental immediate
extension. The extensions $\Gamma|vK=\Gamma|vK(x_1,\ldots,x_{n-1})$ and
$k|Kv=k|K(x_1,\ldots, x_{n-1})v$ are both finite since they are finitely
generated and algebraic. Hence by Proposition~\ref{5.4} there is an
extension of $v$ from $K(x_1,\ldots,x_{n-1})$ to $K(x_1,\ldots,x_n)$
such that $vK(x_1,\ldots,x_n)=\Gamma$ and $K(x_1,\ldots,x_n)v=k$.   \QED

\parm
Now we give the
\pars
{\bf Proof of Theorem~\ref{conv2}}: \
Corollary~\ref{fingentb} shows that $n\geq\rho+\tau$. If $n=\rho+\tau$,
then Corollary~\ref{fingentb} tells us that $vF|vK$ and $Fv|Kv$ are
finitely generated extensions. The fact that $vF|vK$ and $Fv|Kv$ are
always countably generated follows from Theorem~\ref{count} by induction
on the transcendence degree $n$.

Suppose that $F$ admits a transcendence basis $x_1,\ldots,x_n$ such that
the residue field extension
$Fv|K(x_1,\ldots,x_{n-1})v$ is of transcendence degree 1. (This is in
particular the case when $n=\tau$.) Then by Ohm's Ruled Residue Theorem,
$Fv$ is a rational function field in one variable over a finite
extension of $K(x_1,\ldots,x_{n-1})v$. Whenever equality holds in
(\ref{wtdgeq}) for an extension $L|K$ of finite transcendence degree, it
will hold in the respective inequality for any subextension of $L|K$.
Hence if $n=\rho+\tau$ then by Corollary~\ref{fingentb}, $K(x_1,\ldots,
x_{n-1})v|Kv$ is a finitely generated field extension, and if $n=\tau$,
this extension is of transcendence degree $\tau-1$. This yields that B3)
holds for $k=Fv$.

Now suppose that there does not exist such a transcendence basis, and
that $n=\rho+\tau$ with $\rho=1$ and $v$ is trivial on $K$. We pick any
transcendence basis $x_1,\ldots,x_n\,$. Since $Fv|K(x_1,\ldots,x_{n-1})v$
is algebraic and $n=\rho+\tau$, we must have that the rational rank of
$vF/vK(x_1,\ldots,x_{n-1})$ is 1. Since $\rho=1$ and $v$ is trivial on
$K$, this means that $v$ is also trivial on $K(x_1,\ldots,x_{n-1})$.
Hence, $Kv=K$ and $K(x_1,\ldots,x_{n-1})v=K(x_1,\ldots,x_{n-1})$ (modulo
an isomorphism). Remark~\ref{allextrem} now shows
that $Fv$ is a simple algebraic extension of $K(x_1,\ldots,x_{n-1})$,
which in turn is a rational function field in $\tau$ variables over
$K=Kv$, or equal to $Kv$ if $\tau=0$. Together with what we have shown
before, this proves that B2) holds for $k=Fv$.

To conclude this proof, assume that $\rho=0=\tau$. Then by
Lemma~\ref{a-i}, $(F.\tilde{K}| \tilde{K},v)$ is immediate, for
any extension of $v$ from $F$ to $F.\tilde{K}$. This shows that
$(\tilde{K},v)$ admits an immediate extension of transcendence
degree $n$.                                              \QED

\parm
Finally, we give the
\pars
{\bf Proof of Theorem~\ref{extord}}: \
We have to show that whenever we construct an extension of the form
$(L_2|L_1,v)$ with $Kv\subseteq L_1v\subseteq L_2v\subseteq k$ in the
previous proof, then the ordering of $L_2v$ induced by the given
ordering of $k$ can be lifted to an extension of the lifting that we
have already obtained on $L_1\,$. Whenever we apply
Propositions~\ref{finhcf}, \ref{5.4} and ~\ref{cgcg}, we obtain this
already from the assertions of these propositions. Whenever we apply
Lemma~\ref{prelBour}, we obtain this from Proposition~\ref{<ext}
because in this case, we always have that $vL_2$ is generated over
$vL_1$ by rationally independent values and therefore, $2vL_2\cap
vL_1=2vL_1\,$. Finally, whenever $(L_2|L_1,v)$ is an immediate
extension, we can also apply Proposition~\ref{<ext} because
$2vL_2\cap vL_1=2vL_1\cap vL_1=2vL_1\,$.               \QED

%
%Ä - Ä Ä Ä Ä Ä Ä Ä Ä Ä Ä Ä Ä Ä Ä Ä Ä Ä Ä Ä Ä Ä Ä Ä Ä Ä Ä Ä Ä Ä Ä Ä Ä Ä
%
\section{Homogeneous sequences}              \label{sectkrseq}
In this section, we will develop special sequences which under certain
tameness conditions can be used to determine implicit constant fields.
%
%Since implicit constant fields always contain the henselization of the
%base field, we can assume from the start that the base field is
%henselian. Therefore, {\bf we will from now on always assume that
%$(K,v)$ is henselian.} This simplifies our work, but it should be
%mentioned that everything can be generalized to the case where
%$(K,v)$ is not henselian.

%
%Ä - Ä Ä Ä Ä Ä Ä Ä Ä Ä Ä Ä Ä Ä Ä Ä Ä Ä Ä Ä Ä Ä Ä Ä Ä Ä Ä Ä Ä Ä Ä Ä Ä Ä
%
\subsection{Homogeneous approximations}
Let $(K,v)$ be any valued field and $a,b$ elements in some
valued field extension $(L,v)$ of $(K,v)$. We will say that $a$ is
\bfind{strongly homogeneous over $(K,v)$} if $a\in K\sep\setminus K$,
the extension of $v$ from $K$ to $K(a)$ is unique (or equivalently,
$K(a)|K$ is linearly disjoint from all henselizations of $(K,v)$), and
\begin{equation}                            \label{stronghom}
va\;=\;\mbox{\rm kras}(a,K)\;.
\end{equation}
Note that in this case, $va=v(\sigma a -a)$ for all automorphisms such
that $\sigma a\ne a$; indeed, we have that $v\sigma a=va$ and therefore,
$va\geq v(\sigma a -a)\geq va$.

We will say that $a$ is \bfind{homogeneous over $(K,v)$} if there is
some $d\in K$ such that $a-d$ is strongly homogeneous over $(K,v)$,
i.e.,
\[v(a-d)\;=\;\mbox{\rm kras}(a-d,K)\;=\;\mbox{\rm kras}(a,K)\;.\]
We call $a\in L$ a \bfind{homogeneous approximation of $b$
over $K$} if $a$ is homogeneous over $K$ and $v(b-a)>vb$ (and
consequently, $va=vb$). From Corollary~\ref{vkras} we obtain:

\begin{lemma}                               \label{krlkra}
If $a\in L$ is a homogeneous approximation of $b$ then $a$ lies in the
henselization of $K(b)$ w.r.t.\ every extension of the valuation $v$
from $K(a,b)$ to $\widetilde{K(b)}$.
\end{lemma}

We will also exploit the following easy observation:
\begin{lemma}                               \label{homup}
Let $(K',v)$ be any henselian extension field of $(K,v)$ such that
$a\notin K'$. If $a$ is homogeneous over $(K,v)$, then it is also
homogeneous over $(K',v)$, and $\mbox{\rm kras}(a,K)=\mbox{\rm kras}
(a,K')$. If $a$ is strongly homogeneous over $(K,v)$, then it is also
strongly homogeneous over $(K',v)$.
\end{lemma}
\begin{proof}
Suppose that $a-d$ is separable-algebraic over $K$ and $v(a-d)=
\mbox{\rm kras}(a-d,K)$ for some $d\in K$. Then $a-d$ is also
separable-algebraic over $K'$. Further, $\mbox{\rm kras}(a-d,K')\leq
\mbox{\rm kras}(a-d,K)$ since restriction to $\tilde{K}$ is a map
sending $\{\sigma\in \Gal K'\mid \sigma a\ne a\}$ into $\{\sigma\in\Gal
K\mid \sigma a\ne a\}$. Hence, $v(a-d)\leq \mbox{\rm kras}(a-d,K')\leq
\mbox{\rm kras}(a-d,K)= v(a-d)$, which shows that equality holds
everywhere. Thus, $\mbox{\rm kras}(a,K')=\mbox{\rm kras}(a-d,K')=
\mbox{\rm kras}(a-d,K)= \mbox{\rm kras}(a,K)$. Since $(K',v)$ is
henselian by assumption, the extension of $v$ from $K'$ to $K'(a)$ is
unique. This shows that $a-d$ is strongly homogeneous over $(K',v)$, and
concludes the proof of our assertions.
\end{proof}

The following gives the crucial criterion for an element to be
(strongly) homogeneous over $(K,v)$:
\begin{lemma}                              \label{lstronghom}
Suppose that $a\in\tilde{K}$ and that there is some extension of $v$
from $K$ to $K(a)$ such that if {\rm e} is the least positive integer
for which ${\rm e}va\in vK$, then
\sn
a) \ {\rm e} is not divisible by $\chara Kv$,\n
b) \ there exists some $c\in K$ such that $vca^{\rm e}=0$, $ca^{\rm e}v$
is separable-algebraic over $Kv$, and the degree of $ca^{\rm e}$ over
$K$ is equal to the degree {\rm f} of $ca^{\rm e}v$ over $Kv$.
\sn
Then $[K(a):K]={\rm ef}$ and if $a\notin K$, then $a$ is strongly
homogeneous over $(K,v)$.
\end{lemma}
\begin{proof}
We have
\begin{eqnarray*}
{\rm ef} & \geq & [K(a):K(a^{\rm e})]\cdot [K(a^{\rm e}):K]\>=\>
[K(a):K] \\
 & \geq & (vK(a):vK)\cdot [K(a)v:Kv]\>\geq\>{\rm ef}\;.
\end{eqnarray*}
So equality holds everywhere, and we obtain $[K(a):K]={\rm ef}$,
$(vK(a):vK)={\rm e}$ and $[K(a)v:Kv]={\rm f}$. By the fundamental
inequality, the latter implies that the extension of $v$ from $K$ to
$K(a)$ is unique.

Now assume that $a\notin K$. Take two distinct conjugates $\sigma a\ne
\tau a$ of $a$ and set $\eta:= \sigma a/\tau a\ne 1$. If $\sigma a^e \ne
\tau a^e$, then $c\sigma a^e= \sigma ca^e$ and $c\tau a^e= \tau ca^e$
are distinct conjugates of $ca^e$. By hypothesis, their residues are
also distinct and therefore, the residue of $\sigma a^e/ \tau a^e=
\eta^e$ is not $1$. It follows that the residue of $\eta$ is not $1$. If
$\sigma a^e = \tau a^e$, then $\eta$ is an e-th root of unity. Since e
is not divisible by the residue characteristic, it again follows that
the residue of $\eta$ is not equal to $1$. Hence in both cases, we
obtain that $v(\eta-1)=0$, which shows that $v(\sigma a - \tau a)= v\tau
a=va$. We have now proved (\ref{stronghom}).
\end{proof}

\begin{lemma}                               \label{exkrap}
Assume that $b$ is an element in some algebraically closed valued field
extension $(L,v)$ of $(K,v)$. Suppose that there is some ${\rm e}\in\N$
not divisible by $\chara Kv$, and some $c\in K$ such that $vcb^{\rm e}
=0$ and $cb^{\rm e}v$ is separable-algebraic over $Kv$. If ${\rm e}>0$
or $cb^{\rm e}v\notin Kv$, then we can find in $L$ a strongly
homogeneous approximation of $b$ over $K$.
\end{lemma}
\begin{proof}
Take a monic polynomial $g$ over $K$ with $v$-integral coefficients
whose reduction modulo $v$ is the minimal polynomial of $cb^{\rm e}v$
over $Kv$. Then let $d\in\tilde{K}$ be the root of $g$ whose residue is
$cb^{\rm e}v$. The degree of $d$ over $K$ is the same as that of
$cb^{\rm e}v$ over $Kv$. We have that $v(\frac{d}{cb^{\rm e}}-1)>0$.
So there exists $a_0\in \tilde{K}$ with residue $1$ and such that
$a_0^{\rm e}= \frac{d}{cb^{\rm e}}$. Then for $a:=a_0b$, we find that
$v(a-b)=vb+ v(a_0-1)>vb$ and $ca^{\rm e}=d$. By the foregoing lemma,
this shows that $a$ is a strongly homogeneous approximation of $b$.
\end{proof}

%
%Ä - Ä Ä Ä Ä Ä Ä Ä Ä Ä Ä Ä Ä Ä Ä Ä Ä Ä Ä Ä Ä Ä Ä Ä Ä Ä Ä Ä Ä Ä Ä Ä Ä Ä
%
\subsection{Homogeneous sequences}
Let $(K(x)|K,v)$ be any extension of valued fields. We fix an extension
of $v$ to $\widetilde{K(x)}$.

\pars
Let $S$ be an initial segment of $\N$, that is, $S=\N$ or
$S=\{1,\ldots,n\}$ for some $n\in\N$ or $S=\emptyset$. A sequence
\[{\eu S}\;:=\;(a_i)_{i\in S}\]
of elements in $\tilde{K}$  will be called a \bfind{homogeneous sequence
for $x$} if the following conditions are satisfied for all $i\in S$
(where we set $a_0:=0$):
\sn
{\bf (HS)} \ \ $a_i-a_{i-1}$ is a homogeneous approximation of
$x-a_{i-1}$ over $K(a_0,\ldots,a_{i-1})$.
\sn
Recall that then by definition of ``homogeneous'', $a_i\notin
K(a_0,\ldots,a_{i-1})^h$. We call $S$ the \bfind{support} of the
sequence $\eu S$. We set
\[K_{\eu S}\;:=\;K(a_i\mid i\in S)\;.\]
If ${\eu S}$ is the empty sequence, then $K_{\eu S}=K$.

\pars
From this definition, we obtain:
\begin{lemma}                               \label{kspcs}
If $1\leq i<j\in S$, then
\begin{equation}                            \label{pcs}
v(x-a_j)\> >\>v(x-a_i)\>=\>v(a_{i+1}-a_i)\;.
\end{equation}
If $S=\N$ then $(a_i)_{i\in S}$ is a pseudo Cauchy sequence with pseudo
limit $x$.
\end{lemma}
\begin{proof}
If $1\leq i\in S$, then $a_i-a_{i-1}$ is a homogeneous approximation of
$x-a_{i-1}\,$. Hence by definition,
\[v(x-a_i)\>=\>v(x-a_{i-1}-(a_i-a_{i-1}))\> >\>v(x-a_{i-1})\;,\]
whence $v(a_i-a_{i-1})=\min\{v(x-a_i),v(x-a_{i-1})\}=v(x-a_{i-1})$.
If $i<j\in S$, then by induction, $v(x-a_j)>v(x-a_i)$.

Suppose that $S=\N$. Then it follows by induction that for all
$k>j>i\geq 1$,
\[v(x-a_k)\> >\>v(x-a_j) \> >\>v(x-a_i)\]
and therefore,
\begin{eqnarray*}
v(a_k-a_j) & = & \min\{v(x-a_k),v(x-a_j)\}\>=\>v(x-a_j)> v(x-a_i)\\
 & = & \min\{v(x-a_j),v(x-a_i)\}\>=\>v(a_j-a_i)\;.
\end{eqnarray*}
This shows that $(a_i)_{i\in S}$ is a pseudo Cauchy sequence. The
equality in (\ref{pcs}) shows that $x$ is a pseudo limit of this sequence.
\end{proof}

Let us also observe the following:
\begin{lemma}                               \label{xx}
Take $x,x'\in L$. If $a\in L$ is a homogeneous approximation of $x$ over
$K$ and if $v(x-x')>vx$, then $a$ is also a homogeneous approximation of
$x'$ over $K$. If $(a_i)_{i\in S}$ is a homogeneous sequence for $x$
over $K$ and if $v(x-x')>v(x-a_k)$ for all $k\in S$, then $(a_i)_{i\in
S}$ is also a homogeneous sequence for $x'$ over $K$.

In particular, for each $k\in S$ such that $k>1$, $(a_i)_{i<k}$ is a
homogeneous sequence for $a_k$ over $K$.
\end{lemma}
\begin{proof}
Suppose that $a$ is a homogeneous approximation of $x$ over $K$. Then
$v(x-a)>vx$. If also $v(x-x')>vx$, then $vx'=vx$ and $v(x'-a)\geq
\max\{v(x-x'),v(x-a)\}>vx=vx'$. This yields the first assertion.

Now assume that $(a_i)_{i\in S}$ is a homogeneous sequence for $x$
over $K$ and that $v(x-x')>v(x-a_k)$ for all $k\in S$. Then for all
$k\in S$, $v(x'-a_k)=\min\{v(x'-x),v(x-a_k)\}=v(x-a_k)$. Hence,
\begin{eqnarray*}
v(x'-a_{k-1}-(a_k-a_{k-1})) & = & v(x'-a_k)\>=\>v(x-a_k)\\
 & > & v(x-a_{k-1})\>=\>v(x'-a_{k-1})
\end{eqnarray*}
showing that $a_k-a_{k-1}$ is also a homogeneous approximation of
$x'-a_{k-1}\,$. Hence, $(a_i)_{i\in S}$ is a homogeneous sequence for
$x'$ over $K$.
\end{proof}

What is special about homogeneous sequences is described by the
following lemma:
\begin{lemma}                               \label{KSinL}
Assume that $(a_i)_{i\in S}$ is a homogeneous sequence for $x$ over $K$.
Then
\begin{equation}                            \label{KSin}
K_{\eu S} \>\subset\> K(x)^h \;.
\end{equation}
For every $k\in S$, $a_1,\ldots,a_k\in K(a_k)^h$. If
$S=\{1,\ldots,n\}$, then
\begin{equation}                            \label{in}
K_{\eu S}^h \>=\> K(a_n)^h\;.
\end{equation}
\end{lemma}
\begin{proof}
We use induction on $k\in S$. Suppose that we have already shown that
$a_{k-1}\>\in\>K(x)^h$. As $a_k-a_{k-1}$ is a homogeneous approximation
of $x-a_{k-1}\,$, we know from Lemma~\ref{krlkra} that
\[a_k-a_{k-1}\>\in\>K(x-a_{k-1})^h \subseteq K(x)^h\;.\]
This proves (\ref{KSin}). Now all other assertions follow when we
replace $x$ by $a_k$ in the above argument, using the fact that
by the previous lemma, $(a_i)_{i<n}$ is a homogeneous sequence for $a_n$
over $K$.
\end{proof}

\begin{proposition}                         \label{SNpure}
Assume that ${\eu S}=(a_i)_{i\in S}$ is a homogeneous sequence for
$x$ over $K$ with support $S=\N$. Then $(a_i)_{i\in\N}$ is a pseudo
Cauchy sequence of transcendental type in $(K_{\eu S},v)$ with pseudo
limit $x$, and $(K_{\eu S}(x)|K_{\eu S},v)$ is immediate and pure.
\end{proposition}
\begin{proof}
By Lemma~\ref{kspcs}, $(a_i)_{i\in\N}$ is a pseudo Cauchy sequence with
pseudo limit $x$. Suppose it were of algebraic type. Then by [KA],
Theorem~3, there would exist some algebraic extension $(K_{\eu S}
(b)|K_{\eu S},v)$ with $b$ a pseudo limit of the sequence. But then
$v(x-b) >v(x-a_k)$ for all $k\in S$ and by Lemma~\ref{xx}, $(a_i)_{i\in
S}$ is also a homogeneous sequence for $b$ over $K$. Hence by
Lemma~\ref{KSinL}, $K_{\eu S}^h \subset K(b)^h=K^h(b)$. Since $b$ is
algebraic over $K$, the extension $K^h(b)|K^h$ is finite. On the other
hand, $K_{\eu S}^h|K^h$ is infinite since by the definition of
homogeneous elements, $a_k\notin K(a_i\mid 1\leq i< k)^h$ for every
$k\in\N$ and therefore, each extension $K(a_i\mid 1\leq i \leq k+1)^h|
K(a_i\mid 1\leq i \leq k)^h$ is non-trivial. This contradiction shows
that the sequence is of transcendental type. Hence by definition,
$(K_{\eu S}(x)|K_{\eu S} ,v)$ is pure. Further, it follows from
Lemma~\ref{limtpcs} that $(K_{\eu S}(x)| K_{\eu S},v)$ is immediate.
\end{proof}

This proposition leads to the following definition. A homogeneous
sequence $\eu S$ for $x$ over $K$ will be called \bfind{(weakly) pure
homogeneous sequence} if $(K_{\eu S}(x)|K_{\eu S},v)$ is (weakly) pure
in $x$. Hence if $S=\N$, then $\eu S$ is always a pure homogeneous
sequence. The empty sequence is a (weakly) pure homogeneous sequence for
$x$ over $K$ if and only if already $(K(x)|K,v)$ is (weakly) pure in
$x$.

\begin{theorem}                             \label{thIC}
Suppose that $\eu S$ is a (weakly) pure homogeneous sequence for $x$
over $K$. Then
\[K_{\eu S}^h \;=\; \ic (K(x)|K,v)\;.\]
Further, $K_{\eu S}v$ is the relative algebraic closure of $Kv$ in
$K(x)v$, and the torsion subgroup of $vK(x)/vK_{\eu S}$ is finite.
If $\eu S$ is pure, then $vK_{\eu S}$ is the relative divisible
closure of $vK$ in $vK(x)$.
\end{theorem}
\begin{proof}
The assertions follow from Lemma~\ref{purevgrf}, Lemma~\ref{algtransc}
and Lemma~\ref{icpure}, together with the fact that because $K_{\eu S}
|K$ is algebraic, the same holds for $vK_{\eu S}|vK$ and $K_{\eu S}v|Kv$
by Lemma~\ref{fin}.
\end{proof}

%
%Ä - Ä Ä Ä Ä Ä Ä Ä Ä Ä Ä Ä Ä Ä Ä Ä Ä Ä Ä Ä Ä Ä Ä Ä Ä Ä Ä Ä Ä Ä Ä Ä Ä Ä
%
\subsection{Conditions for the existence of homogeneous
sequences}                                             \label{sectcond}
Now we have to discuss for which extensions $(K(x)|K,v)$ there exist
homogeneous sequences.

An algebraic extension $(L|K,v)$ of henselian fields is called
\bfind{tame} if the following conditions hold:
\sn
(TE1) \ $Lv|Kv$ is separable,\n
(TE2) \ if $\chara Kv=p>0$, then the order of each element in $vL/vK$
is prime to $p$,\n
(TE3) \ $[K':K]=(vK':vK)[K'v:Kv]$ holds for every finite subextension
$K'|K$ of $L|K$.
\sn
Condition (TE3) means that equality holds in the fundamental inequality
(\ref{fiq}). If $L'|K$ is any subextension of $L|K$, then $(L|K,v)$ is a
tame extension if and only if $(L|L',v)$ and $(L'|K,v)$ are (this is
easy to prove if $L|K$ is finite). Further, it is well known that for
$(K,v)$ henselian, the ramification field of the extension $(K\sep|K,v)$
is the unique maximal tame extension of $(K,v)$ (cf.\ [E]). A henselian
valued field $(K,v)$ is called a \bfind{tame field} if all its algebraic
extensions are tame, or equivalently, the following conditions hold:
\sn
(T1) \ $Kv$ is perfect,\n
(T2) \ if $\chara Kv=p>0$, then $vK$ is $p$-divisible,\n
(T3) \ for every finite extension $K'|K$, $\;[K':K]=(vK':vK)[K'v:Kv]$.
\sn
Note that every valued field with a residue field of characteristic zero
is tame; this is a consequence of the Lemma of Ostrowski (cf.\ [R]). It
follows directly from the definition together with the multiplicativity
of ramification index and inertia degree that every finite extension of
a tame field is again a tame field. If $(K,v)$ is a tame field, then
condition (T3) shows that $(K,v)$ does not admit any proper immediate
algebraic extensions; hence by Theorem~3 of [KA], every pseudo Cauchy
sequence in $(K,v)$ without a pseudo limit in $K$ must be of
transcendental type.

If an element $a\in \tilde{K}$ satisfies the conditions of
Lemma~\ref{lstronghom}, then $(K(a)|K,v)$ is a tame extension. The
following implication is also true, as was noticed by Sudesh K.\
Khanduja (cf.\ [KH11], Theorem 1.2):
\begin{proposition}                         \label{hit}
Suppose that $(K,v)$ is henselian. If $a$ is homogeneous over $(K,v)$,
then $(K(a)|K,v)$ is a tame extension. If $\eu S$ is a homogeneous
sequence over $(K,v)$, then $K_{\eu S}$ is a tame extension of $K$.
\end{proposition}
\begin{proof}
Since $K(a-d)=K(a)$ for $d\in K$, we may assume w.l.o.g.\ that $a$ is
strongly homogeneous over $(K,v)$. If $(K(a)|K,v)$ were not a tame
extension, then $a$ would not lie in the ramification field of the
extension $(K\sep|K,v)$. So there would exist an automorphism $\sigma$
in the ramification group such that $\sigma a\ne a$. But by the
definition of the ramification group,
\[v(\sigma a- a)\;>\;va \;=\; \mbox{\rm kras}(a,K)\;,\]
a contradiction.

The second assertion is proved using the first assertion and the fact
that a (possibly infinite) tower of tame extensions is itself a tame
extension.
\end{proof}

\n
In fact, it can also be shown that $va\,=\,\mbox{\rm kras}(a,K)$
implies that $a$ satisfies the conditions of Lemma~\ref{lstronghom}.

\pars
We can give the following characterization of elements in tame
extensions:
\begin{proposition}
An element $b\in \tilde{K}$ belongs to a tame extension of the henselian
field $(K,v)$ if and only if there is a finite homogeneous sequence
$a_1,\ldots,a_k$ for $a$ over $(K,v)$ such that $b\in K(a_k)$.
\end{proposition}
\begin{proof}
Suppose that such a sequence exists. By the foregoing proposition,
$K_{\eu S}$ is a tame extension of $K$. Since $K(a_1,\ldots,a_k)=K(a_k)$
by Lemma~\ref{KSinL}, it contains $b$.

For the converse, let $b$ be an element in some tame extension of
$(K,v)$. Then $b$ satisfies the assumptions of Lemma~\ref{exkrap} and
hence there is a homogeneous approximation $a_1\in\tilde{K}$ of $b$ over
$K$. By the foregoing proposition, $K(a_1)$ is a tame extension of $K$
and therefore, by the general facts we have noted following the
definition of tame extensions, $K(a_1, b-a_1)$ is a tame extension of
$K(a_1)$. We repeat this step, replacing $b$ by $b-a_1\,$. By induction,
we build a homogeneous sequence for $b$ over $K$. It cannot be infinite
since $b$ is algebraic over $K$ (cf.\ Proposition~\ref{SNpure}). Hence
it stops with some element $a_k\,$. Our construction shows that this can
only happen if $b\in K(a_1,\ldots,a_k)=K(a_k)$.
\end{proof}

\begin{proposition}                         \label{tame}
Assume that $(K,v)$ is a henselian field. Then $(K,v)$ is a tame field
if and only if for every element $x$ in any extension $(L,v)$ of $(K,v)$
there exists a weakly pure homogeneous sequence for $x$ over $K$,
provided that $x$ is transcendental over $K$.
\end{proposition}
\begin{proof}
First, let us assume that $(K,v)$ is a tame field and that $x$ is an
element in some extension $(L,v)$ of $(K,v)$, transcendental over $K$.
We set $a_0=0$. We assume that $k\geq 0$ and that $a_i$ for $i\leq k$
are already constructed. Like $K$, also the finite extension $K_k:=
K(a_0,\ldots,a_k)$ is a tame field. Therefore, if $x$ is the pseudo
limit of a pseudo Cauchy sequence in $K_k\,$, then this pseudo Cauchy
sequence must be of transcendental type, and $K_k(x)|K_k$ is pure and
hence weakly pure in $x$.

If $K_k(x)|K_k$ is weakly pure in $x$, then we take $a_k$ to be the last
element of ${\eu S}$ if $k>0$, and ${\eu S}$ to be empty if $k=0$.

Assume that this is not the case. Then $x$ cannot be the pseudo limit of
a pseudo Cauchy sequence without pseudo limit in $K_k\,$. So the set
$v(x-a_k- K_k)$ must have a maximum, say $x-a_k-d$ with $d\in K_k$.
Since we assume that $K_k(x)|K_k$ is not weakly pure in $x$, there exist
${\rm e}\in\N$ and $c\in K_k$ such that $vc(x-a_k-d)^{\rm e}=0$ and
$c(x-a_k-d)^{\rm e}$ is algebraic over $K_kv$. Conditions (T1) and (T2)
that e can be chosen to be prime to $\chara Kv$ and that
$c(x-a_k-d)^{\rm e}v$ is separable-algebraic over $K_kv$. Since
$v(x-a_k-d)$ is maximal in $v(x-a_k- K_k)$, we must have that
${\rm e}>1$ or $c(x-a_k-d)^{\rm e}v\notin K_kv$.

Now
Lemma~\ref{exkrap} shows that there exists a homogeneous approximation
$a\in\tilde{K}$ of $x-a_k-d$ over $K_k\,$; so $a+d$ is a homogeneous
approximation of $x-a_k$ over $K_k$, and we set $a_{k+1}:=a_k+a+d$. This
completes our induction step. If our construction stops at some $k$,
then $K_k(x)|K_k$ is weakly pure in $x$ and we have obtained a weakly
pure homogeneous sequence. If the construction does not stop, then
$S=\N$ and the obtained sequence is pure homogeneous.

\pars
For the converse, assume that $(K,v)$ is not a tame field. We
choose an element $b\in\tilde{K}$ such that $K(b)|K$ is not a tame
extension. On $K(b,x)$ we take the valuation $v_{b,\gamma}$ with
$\gamma$ an element in some ordered abelian group extension such that
$\gamma>vK$. Choose any extension of $v$ to $\tilde{K}(x)$. Since $vK$
is cofinal in in $v\tilde{K}$, we have that $\gamma>v\tilde{K}$. Since
$b\in \tilde{K}$, we find $\gamma\in v\tilde{K}(x)$. Hence,
$(\tilde{K}(x)|\tilde{K},v)$ is value-transcendental.

Now suppose that there exists a weakly pure homogeneous sequence
${\eu S}$ for $x$ over $K$. By Lemma~\ref{aavavtrt}, also $(K_{\eu S}(x)
|K_{\eu S},v)$ is value-transcendental. Since $(K_{\eu S}(x) |K_{\eu S},
v)$ is also weakly pure, it follows that there must be some $c\in
K_{\eu S}$ such that $x-c$ is a value-transcendental element (all
other cases in the definition of ``weakly pure'' lead to immediate or
residue-transcendental extensions). But if $c\ne b$ then $v(b-c)\in
v\tilde{K}$ and thus, $v(c-b)<\gamma$. This implies
$v(x-c)=\min\{v(x-b),v(b-c)\}=v(b-c) \in v\tilde{K}$, a contradiction.
This shows that $b=c\in K_{\eu S}$. On the other hand, $K_{\eu S}$ is a
tame extension of $K$ by Proposition~\ref{hit} and cannot contain $b$.
This contradiction shows that there cannot exist a weakly pure
homogeneous sequence for $x$ over~$K$.
\end{proof}

%
%Ä - Ä Ä Ä Ä Ä Ä Ä Ä Ä Ä Ä Ä Ä Ä Ä Ä Ä Ä Ä Ä Ä Ä Ä Ä Ä Ä Ä Ä Ä Ä Ä Ä Ä
%
\section{Applications}                      \label{sectappl}
Let us show how to apply our results to power series fields. We denote
by $k((G))$ the field of power series with coefficients in the field $k$
and exponents in the ordered abelian group $G$.
\begin{theorem}                             \label{powser}
Let $(K,v)$ be a henselian subfield of a power series field $k((G))$
such that $v$ is the restriction of the canonical valuation of $k((G))$.
Suppose that $K$ contains all monomials of the form $ct^\gamma$ for
$c\in Kv \subseteq k$ and $\gamma\in vK\subseteq \Gamma$. Consider a
power series
\[z\>=\>\sum_{i\in\N} c_it^{\gamma_i}\]
where $(\gamma_i)_{i\in\N}$ is a strictly increasing sequence in $G$,
all $c_i\in k$ are separable-algebraic over $Kv$, and for each $i$
there is an integer ${\rm e}_i>0$ prime to the characteristic of
$Kv$ and such that ${\rm e}_i\gamma_i\in vK$. Then, upon taking
the henselization of $K(z)$ in $k((G))$, we obtain that
$K(c_it^{\gamma_i}\mid i\in\N)\subseteq K(z)^h$. Consequently, $vK(z)$
contains all $\gamma_i\,$, and if $\gamma_i\in vK$ for all $i\in\N$,
then $K(z)v$ contains all $c_i\,$.

If $vK+\sum_{i=1}^{\infty}\Z\gamma_i/vK$ or $Kv(c_i\mid i\in\N)|Kv$ is
infinite, then $z$ is transcendental over $K$, we have $\ic
(K(z)|K,v)=K(c_it^{\gamma_i}\mid i\in\N)$, and
\sn
a) \ $vK(z)$ is the group generated over $vK$ by the elements
$\gamma_i\,$,
\sn
b) \ if $\gamma_i\in vK$ for all $i\in\N$, then $K(z)v=Kv(c_i\mid
i\in\N)$.
\end{theorem}
\begin{proof}
We derive a homogeneous sequence ${\eu S}$ from $z$ as follows. We set
$a_0=0$. If all $c_i$ are in $Kv$ and all $\gamma_i$ are in $vK$, then
we take ${\eu S}$ to be the empty sequence. Otherwise, having chosen
$i_j\in\N$ and defined
\[a_j \>:=\>\sum_{1\leq i\leq i_j} c_it^{\gamma_i}\]
for all $j<m$, we proceed as follows. We let $\ell$ be the first
index in the power series $z-a_{m-1}$ for which $c_\ell t^{\gamma_\ell}
\notin K(a_1, \ldots,a_{m-1})$; if such an index does not exist, we let
$a_{m-1}$ be the last element of ${\eu S}$. Otherwise, we set $i_m:=
\ell$ and $\displaystyle a_m:=\sum_{1\leq i\leq i_m} c_it^{\gamma_i}$.
We have that
\[c_\ell t^{\gamma_\ell}\>=\> a_m-a_{m-1}-d \ \ \mbox{\ \ with \ \ } \ \
d\>=\>\sum_{i_{m-1}< i<\ell} c_it^{\gamma_i}\;\in\; K(a_1,\ldots,
a_{m-1})\;.\]
By assumption, ${\rm e}_\ell vc_\ell t^{\gamma_\ell} = {\rm e}_\ell
\gamma_\ell\in vK$ and hence, $c:=t^{-{\rm e}_\ell\gamma_\ell}\in
K$. We have that $c(c_\ell t^{\gamma_\ell})^{{\rm e}_\ell}=
c_\ell^{{\rm e}_\ell}$. Since $c_\ell$ is separable-algebraic over $Kv$,
the same holds for $c_\ell^{{\rm e}_\ell}$. Since $v$ is trivial on $k$,
the degree of $c_\ell^{{\rm e}_\ell}v$ over $Kv$ is equal to that of
$c_\ell^{{\rm e}_\ell}$ over $K$. It now follows by
Lemma~\ref{lstronghom} that $c_\ell t^{\gamma_\ell}$ is strongly
homogeneous over $K$. By Lemma~\ref{homup} it follows that it is also
strongly homogeneous over the henselian field $K(a_1, \ldots, a_{m-1})$.
Therefore, $a_m-a_{m-1}$ is homogeneous over $K(a_1, \ldots, a_{m-1})$.
Further, $v(z-a_{m-1}-(a_m-a_{m-1}))=v(z-a_m)=\gamma_{\ell+1}>
\gamma_{\ell}\geq v(z-a_{m-1})$. This proves that $a_m-a_{m-1}$ is a
homogeneous approximation of $z-a_{m-1}$ over $K(a_0,\ldots,a_{m-1})$.
By induction, we obtain a homogeneous sequence ${\eu S}$ for $z$ in
$k((G))$. It now follows from Lemma~\ref{KSinL} that
$K(c_it^{\gamma_i}\mid i\in\N)=K(a_j\mid j\in S) \subseteq K(z)^h$;
thus, $\gamma_i= vc_it^{\gamma_i}\in vK(z)$ for all $i\in\N$. If
$\gamma_i\in vK$ and hence $t^{\gamma_i}\in K$, then $c_i\in K(z)^h$;
since the residue map is the identity on elements of $k$, this implies
that $c_i\in K(z)v$.

\pars
Now assume that $vK+\sum_{i=1}^{\infty}\Z\gamma_i/vK$ or $Kv(c_i\mid
i\in\N)|Kv$ is infinite. Then ${\eu S}$ must be infinite, and it follows
from Proposition~\ref{SNpure} that $z$ is a pseudo limit of the pseudo
Cauchy sequence
$(a_i)_{i\in\N}$ of transcendental type. Thus, $z$ is transcendental
over $K$ by Lemma~\ref{limtpcs}. Theorem~\ref{thIC} now shows that $\ic
(K(z)|K,v)=K(c_it^{\gamma_i}\mid i\in\N)$, $vK(z)=vK(c_it^{\gamma_i}
\mid i\in\N)$ and $K(z)v= K(c_it^{\gamma_i} \mid i\in\N)v$. Since
$K(c_it^{\gamma_i}\mid i\in\N)\subseteq K(c_i,t^{\gamma_i}\mid i\in\N)$,
for the proof of assertion a) it now suffices to show that the value
group of the latter field is generated over $vK$ by the $\gamma_i\,$,
and that its residue field is generated over $Kv$ by the $c_i\,$. As
$t^{\gamma}\in K$ for every $\gamma\in vK$, we have that
\[(vK+\sum_{i=1}^{\ell}\Z \gamma_i:vK)\>=\>[K(t^{\gamma_i}\mid
1\leq i\leq\ell):K]\;.\]
On the other hand,
\[[K(t^{\gamma_i}\mid 1\leq i\leq\ell):K] \>\geq\>(vK(t^{\gamma_i}\mid
1\leq i\leq\ell):vK) \geq (vK+\sum_{i=1}^{\ell}\Z \gamma_i:vK)\;,\]
where the last inequality holds since $\gamma_i\in vK(t^{\gamma_i}\mid
1\leq i\leq\ell)$ for $i=1,\ldots,\ell$. We obtain that
$vK(t^{\gamma_i}\mid 1\leq i\leq\ell)=vK+\sum_{i=1}^{\ell}\Z\gamma_i$
and that $K(t^{\gamma_i}\mid 1\leq i\leq\ell)v=Kv$ for all $\ell\in\N$.
This implies that $vK(t^{\gamma_i}\mid i\in\N)= vK+\sum_{i=1}^{\infty}
\Z\gamma_i$ and $K(t^{\gamma_i}\mid i\in\N)v=Kv$.

Set $K':=K(t^{\gamma_i}\mid i\in\N)$. As $K'v=Kv\subset K\subset K'$,
we have that
\[[K'v(c_i\mid 1\leq i\leq\ell):K'v]\>\geq\>
[K'(c_i\mid 1\leq i\leq\ell):K']\;.\]
On the other hand,
\[[K'(c_i\mid 1\leq i\leq\ell):K']\>\geq\>[K'(c_i\mid 1\leq
i\leq\ell)v:K'v]\geq [K'v(c_i\mid 1\leq i\leq\ell):K'v]\;,\]
where the last inequality holds since $c_i\in K'(c_i\mid 1\leq
i\leq\ell)v$ for $i=1,\ldots,\ell$. We obtain that
$K'(c_i\mid 1\leq i\leq\ell)v=K'v(c_i\mid 1\leq i\leq\ell)$ and that
$vK'(c_i\mid 1\leq i\leq\ell)=vK'$. This implies that
$vK(c_i,t^{\gamma_i}\mid i\in\N)= vK'(c_i\mid i\in\N)=
vK'=vK+\sum_{i=1}^{\infty}\Z\gamma_i$ and $K(c_i,t^{\gamma_i}\mid
i\in\N)v=K'(c_i\mid i\in\N)v=K'v(c_i\mid i\in\N)=Kv(c_i\mid i\in\N)$.
This proves assertions a) and b).
\end{proof}

\begin{remark}
Assertions a) and b) of the previous theorem will also hold if $K=
Kv((vK))$. Indeed, if $G_0$ denotes the subgroup of $G$ generated by the
$\gamma_i$ over $Kv$, and $k_0=Kv(c_i\mid i\in\N)$, then $K(z)\subset
k_0((G_0))$ and therefore, $vK(z)\subseteq G_0$ and $K(z)v\subseteq
k_0\,$.
\end{remark}

\parm
Our methods also yield an alternative proof of the following well known
fact:
\begin{theorem}
The algebraic closure of the field $\Q_p$ of p-adic numbers is not
complete.
\end{theorem}
\begin{proof}
Choose any compatible system of $n$-th roots $p^{1/n}$ of $p$, that is,
such that $(p^{1/mn})^m=p^{1/n}$ for all $m,n\in\N$. For $i\in\N$,
choose any $n_i\in\N$ not divisible by $p$, and set $\gamma_i:=i+
\frac{1}{n_i}$ if $n_i>1$, and $\gamma_i=i$ otherwise. Further, choose
$c_i$ in some fixed set of representatives in $\widetilde{\Q_p}$ of its
residue field $\widetilde{\F_p}$ such that the degree of $c_i^{n_i}$
over $\Q_p$ is equal to the degree of $c_i^{n_i} v_p$ over $\Q_p
v_p=\F_p\,$. Then set
\begin{equation}                            \label{b_i}
b_i\;:=\;\sum_{1\leq j\leq i} c_j p^{\gamma_j} \;.
\end{equation}
Since $\gamma_i<\gamma_{i+1}$ for all $i$ and the sequence
$(\gamma_i)_{i\in\N}$ is cofinal in the value group $\Q$ of
$\widetilde{\Q_p}$, the sequence $(b_i)_{i\in\N}$ is a Cauchy sequence
in $(\widetilde{\Q_p},v_p)$. If this field were complete, it would
contain a pseudo limit $z$ to every such Cauchy sequence.
On the other hand, as in the proof of Theorem~\ref{powser} one shows
that $\Q_p(c_i p^{\gamma_i}\mid i\in\N)\subseteq \Q_p(z)^h$, and that
\sn
a) \ $v\Q_p(z)$ is the group generated over $\Z$ by the elements
$\gamma_i\,$,
\sn
b) \ if $\gamma_i\in\Z$ for all $i\in\N$, then $\Q_p(z)v=\F_p(c_i v_p
\mid i\in\N)$,
\sn
c) \ if $v\Q_p(z)/\Z$ or $\F_p(c_i v_p\mid i\in\N)|\F_p$
is infinite, then $z$ is transcendental over $\Q_p\,$.
\sn
Hence, if we choose $(n_i)_{i\in\N}$ to be a strictly increasing
sequence and $c_i=1$ for all $i$, or if we choose $n_i=1$ for all $i$
and the elements $c_i$ of increasing degree over $\Q_p\,$, then $z$ will
be transcendental over $\Q_p\,$. Since $z$ lies in the completion of
$\widetilde{\Q_p}$, we have now proved that this completion is
transcendental over $\widetilde{\Q_p}$.
\end{proof}

With the same method, we can also prove another well known result:
\begin{theorem}
The completion $\C_p$ of $\widetilde{\Q_p}$ admits a pseudo Cauchy
sequence without a pseudo limit in $\C_p\,$. Hence, $\C_p$ is not
maximal and not spherically complete.
\end{theorem}
\begin{proof}
In the same setting as in the foregoing proof, we now choose $c_i=1$
for all $i$. Further, we choose $(n_i)_{i\in\N}$ to be a strictly
increasing sequence and set $\gamma_i:= 1-\frac{1}{n_i}$. Then
$(b_i)_{i\in\N}$ is a pseudo Cauchy sequence. Suppose it would admit
a pseudo limit $y$ in $\C_p\,$. Then, using that $\widetilde{\Q_p}$ is
dense in $\C_p\,$, we could choose $z\in \widetilde{\Q_p}$ such that
$v_p(y-z)\geq 1$. Since $1>\gamma_i$ for all $i$, it would follow that
also $z$ is a pseudo limit of $(b_i)_{i\in\N}\,$. But as in the
foregoing proof one shows that $z$ must be transcendental over $\Q_p\,$.
This contradiction shows that $(b_i)_{i\in\N}$ cannot have a pseudo
limit in $\C_p\,$. By the results of [KA], this implies that $\C_p$
admits a proper immediate extension, which shows that $\C_p$ is not
maximal. With the elements $b_i$ defined as in (\ref{b_i}), we also
find that the intersection of the nest $(\{a\in\C_p\mid v(a-b_i)\geq
\gamma_{i+1}\})_{i\in\N}$ of balls is empty. Hence, $\C_p$ is not
spherically complete.
\end{proof}

\newcommand{\lit}[1]{\bibitem #1{#1}}

\end{document}